\documentclass[11pt]{amsart}
\usepackage{amsmath,amsthm, amscd, amssymb, amsfonts, mathrsfs}
\usepackage[all]{xy}
\usepackage{color}
\usepackage{hhline}

\newcommand{\com}[1]{\textcolor{blue}{\textbf{#1}}}

\newtheorem{lema}{Lemma}[section]
\newtheorem{prop}[lema]{Proposition}
\newtheorem{Cor}[lema]{Corollary}
\newtheorem{Thm}[lema]{Theorem}

\theoremstyle{definition}
\newtheorem{Def}[lema]{Definition}
\newtheorem{definition}[lema]{Definition}
\newtheorem{Ex}[lema]{Example}
\newtheorem{Expls}[lema]{Examples}

\newtheorem*{claim}{Claim}
\newtheorem*{question}{Question}

\theoremstyle{remark}

\newtheorem{Rem}[lema]{Remark}
\newtheorem{Rems}[lema]{Remarks}

\newcommand\car{\operatorname{char}}

\newcommand{\co}{\operatorname{co} }

\newcommand{\vb}{{\vartheta}}

\renewcommand{\_}[1]{_{\left( #1 \right)}}

\newcommand{\cou}{\varepsilon }
\newcommand{\ot}{\otimes}

\def\mT{\mathcal{T}}
\def\mP{\mathcal{P}}

\def\bF{\mathbb{F}}

\newcommand{\J}{{\mathcal J}}

\newcommand{\ydgg}{{}^{\ku \Gamma}_{\ku\Gamma}\mathcal{YD}}

\newcommand{\ku}{\Bbbk}

\newcommand{\Z}{{\mathbb Z}}
\newcommand{\BB}{{\mathbb B}}
\newcommand{\CC}{{\mathbb C}}
\newcommand{\N}{{\mathbb N}}

\newcommand{\Gc}{{\mathcal G}}

\newcommand{\Fg}{{\mathfrak F}}

\newcommand{\D}{{\mathcal D}}

\newcommand{\Ee}{{\mathcal E}}

\newcommand{\cI}{{\mathcal I}}

\newcommand{\prov}{{\mathcal H}}
\newcommand{\cH}{{\mathcal H}}
\newcommand{\prova}{{\mathcal A}}
\newcommand{\cA}{{\mathcal A}}
\newcommand{\T}{{\mathcal T}}
\newcommand{\varep}{\varepsilon }

\newcommand{\toba}{{\mathcal B}}
\newcommand{\B}{{\mathcal B}}

\newcommand{\mL}{{\mathcal L}}

\newcommand{\Pc}{{\mathcal P}}

\newcommand{\ydh}{{}^H_H\mathcal{YD}}

\newcommand{\hyd}{\mathcal{YD}^H_H}
\newcommand{\ydl}{\mathcal{YD}{}^L_L}

\newcommand{\Ss}{{\mathcal S}}

\newcommand{\End}{\operatorname{End}}

\newcommand\card{\operatorname{card}}

\newcommand\ad{\operatorname{ad}}
\newcommand\Hom{\operatorname{Hom}}

\newcommand\id{\operatorname{id}}

\newcommand\op{\operatorname{op}}

\def\pf{\begin{proof}}
\def\epf{\end{proof}}

\numberwithin{equation}{section}\theoremstyle{plain}

\newcommand\Alg{\operatorname{Alg}}

\newcommand\Gal{\operatorname{Gal}}

\newcommand\Cleft{\operatorname{Cleft}}

\newcommand\can{\operatorname{can}}
\newcommand\cop{{\operatorname{cop}}}

\newcommand\gr{\operatorname{gr}}
\newcommand\ord{\operatorname{ord}}

\def\s{\mathbb{S}}

\newcommand\YD[2]{\,^{#1}_{#2}\mathcal{YD}}

\begin{document}

\title[Lifting via cocycles]{Lifting via cocycle deformation}
\author[Andruskiewitsch; Angiono; Garc\'ia Iglesias; Masuoka; Vay]
{Nicol\'as Andruskiewitsch; Iv\'an Angiono; Agust\'in Garc\'ia Iglesias; Akira Masuoka;
Cristian Vay}

\address{N.A., I. A., A.G.I. and C.V.: FaMAF-CIEM (CONICET), Universidad Nacional de C\'ordoba,
Medina A\-llen\-de s/n, Ciudad Universitaria (5000) C\' ordoba, Rep\'
ublica Argentina.} \email{(andrus|angiono|aigarcia|vay)@famaf.unc.edu.ar}

\address{A. M.: Institute of Mathematics,  University of Tsukuba,  Ibaraki 305-8571, Japan.}
\email{akira@math.tsukuba.ac.jp}

\thanks{N.A., I. A., A.G.I. and C.V. were  supported by
 CONICET, ANPCyT and Secyt (UNC). A. M. was supported by a
Grant-in-Aid for   Scientific Research (C) 23540039, JSPS}

\begin{abstract}
We develop a strategy to compute all liftings of a Nichols algebra over a
finite dimensional cosemisimple Hopf algebra. We
produce them as cocycle deformations of the bosonization of these two. In
parallel, we study the shape of any such lifting.
\end{abstract}

\maketitle

\section{Introduction}\label{sect:intro}

Let $A$ be a finite-dimensional Hopf algebra whose coradical is a Hopf
subalgebra $H$. Then the graded algebra associated to the coradical filtration of
$A$ is again a Hopf algebra, which is given by a smash product $\gr
A\simeq R\# H$, for $R=\bigoplus_{n\geq 0}R^n$ a graded Hopf algebra in $\ydh$, the
category of Yetter-Drinfeld modules over $H$. Let $V=R^1$, then the subalgebra of $R$
generated by $V$ is the {\it Nichols
algebra} $\B(V)$ \cite{AS2}; this is a braided Hopf algebra in $\ydh$ which is also
defined for
every $V\in\ydh$ by a universal quotient $T(V)/\J(V)$, for $\J(V)$ an ideal generated by
 homogeneous elements of degree $\geq 2$.

If $\gr A=\B(V)\# H$, then $A$ is called a
\emph{lifting} or \emph{deformation} of $\toba(V)$ (over $H$). Hence, deformations of
$\toba(V)$ give rise to new examples of Hopf algebras. Moreover, there are classes of
Hopf algebras (as pointed Hopf algebras over abelian groups) in which every example arises
as such a deformation.

\subsection{The problem} In this article, we develop a strategy to
compute all the liftings or deformations of a Nichols algebra. More precisely, we consider
\begin{flalign}
\label{eqn:H}&\text{a Hopf algebra $H$ which is finite-dimensional and
cosemisimple;}\\
\label{eqn:V}&\text{$V \in \ydh$ such that $\dim V < \infty$ and $\J(V)$ is finitely generated.}
\end{flalign}
The problem is to describe all Hopf algebras $A$ such that
\begin{align}\label{eqn:main-pbm}
\gr A \simeq \toba(V) \# H.
\end{align}
Notice that the coradical of $A$ is isomorphic to $H$ by \eqref{eqn:main-pbm}, see
\cite{AS2}. This problem is one of the
steps in the
Lifting Method \cite{AS1, AS2}, see also the generalization proposed in \cite{AC}. To
deal with it, we split it into two parts:

\begin{enumerate}\renewcommand{\theenumi}{\alph{enumi}}
\renewcommand{\labelenumi}{(\theenumi)}
  \item\label{part:a} To detect the shape of all possible deformations.
  \item\label{part:b} To show that these proposed deformations actually are so.
\end{enumerate}

Problem \eqref{part:a} is usually taken by examination of the comodule structure of the
first term of the coradical filtration, what would give possible deformations by defining
relations, see Section \ref{sect:shape}.

However it is not apparent that the proposed deformations have the desired property;
namely, such deformation $A$ would bear an epimorphism  $\toba(V) \# H
\twoheadrightarrow\gr A$ but whether this is an isomorphism requires an extra reasoning.
This is Problem \eqref{part:b} and there have been different approaches to face up to it:
the Diamond Lemma \cite{AS1, AG2, AV}; a reduction to the first term of the coradical
filtration followed by some representation theory, assuming that the Nichols algebra is
quadratic \cite{GGI}; a combination of deformation by cocycles and an examination of the
PBW basis \cite{AS3}.

We briefly recall this last approach highlighting some features that are present in
the strategy below; see \emph{loc.~cit.} for more details and undefined notation. There,
$H$ is assumed to be the group algebra of a finite abelian group $\Gamma$
(with some restrictions on the order) and $V \in \ydh$ has a finite-dimensional Nichols
algebra;
therefore, by the restrictions alluded to, $V$ is of Cartan type and gives rise to a
Dynkin diagram $\varDelta$.
The defining ideal $\J(V)$ is generated by three kind of relations:
\begin{enumerate}\renewcommand{\theenumi}{\roman{enumi}}
\renewcommand{\labelenumi}{(\theenumi)}
  \item Serre relations in the same connected component of $\varDelta$,
  \item Serre relations between vertices in different connected components,
  \item powers of root vectors.
\end{enumerate}
It is then shown that in any deformation $A$ the Serre relations in the same connected
component still hold,
and the other relations deform respectively to the so-called linking relations, controlled
by a family of parameters $\boldsymbol\lambda$,
and the so-called power of root vector relations, controlled by a second family of
parameters $\boldsymbol\mu$.
Hence the $A$ should be of the form $u(\D,\boldsymbol\lambda,\boldsymbol\mu) = T(V)\# H /
\J$, where the ideal
$\J$ is generated by:

\begin{enumerate}\renewcommand{\theenumi}{\roman{enumi}}
\renewcommand{\labelenumi}{(\theenumi)}
    \item Serre relations (in the same connected component),
\item linking relations,
\item power of root vector relations.
\end{enumerate}
To show that $u(\D,\boldsymbol\lambda,\boldsymbol\mu)$ has the desired dimension $\dim
\toba(V) \vert
\Gamma\vert$, the procedure in \cite{AS3} goes as follows.

\begin{enumerate}\renewcommand{\theenumi}{\alph{enumi}}
\renewcommand{\labelenumi}{(\theenumi)}
\smallbreak  \item Let $U(\D,\boldsymbol\lambda) = T(V)\# H / \J_0$,  where the ideal
$\J_0$ is
generated by the Serre relations (in the same connected component)
and the linking relations. Then $U(\D,\boldsymbol\lambda)$ has the ``right'' basis; it is
proved by
induction on the number of connected components, via cocycle deformation in the inductive
step.

\smallbreak  \item Finally, $u(\D,\boldsymbol\lambda) = U(\D,\boldsymbol\lambda)/\J_1$,
where $\J_1$ is
generated by the power of root vector relations, has the right dimension by a delicate
argument using centrality of these last relations in $U(\D,\boldsymbol\lambda)$.
\end{enumerate}

\subsection{The background}

The family of Hopf algebras $u(\D,\boldsymbol\lambda,\boldsymbol\mu)$ contains
the liftings of quantum linear
spaces defined in \cite{AS1}. It was shown in \cite{M0} that
these liftings of quantum linear spaces are cocycle deformations of their associated
graded Hopf algebras. Further work in this direction was done in \cite{Di, BDR, GM1}; in
this last paper it was stated that any Hopf algebra
$u(\D,\boldsymbol\lambda,\boldsymbol\mu)$ is a cocycle
deformation of its associated graded Hopf algebra, but the argument had a gap and a
complete proof was given in \cite{M4}.

The result in  \cite{M4} is first extended to the non-abelian
case in \cite{GIM} where it is shown that every finite-dimensional pointed Hopf algebra
$H$ over $\s_3$ or $\s_4$ is again a cocycle deformation of $\gr H$. In
\cite{AV2} it is shown that this is also the case for finite-dimensional copointed Hopf
algebras over $\s_3$. Also, in \cite{GIV} this is shown for some pointed or copointed
Hopf algebras associated to affine racks. In all of these papers the results
are achieved by computing Hopf
biGalois objects. In \cite{GaM}, the authors pick up the work in \cite{GM1} to
explicitly compute cocycles as exponentials of Hochschild 2-cocycles. They show that every
finite-dimensional pointed Hopf algebra $H$ over the dihedral groups $D_{4t}$ is a cocycle
deformation of $\gr H$.

\subsection{The strategy}

In the present paper, we propose to reverse the order and start by computing all cocycle
deformations following ideas in \cite{M4}. Observe that, since a deformation by
cocycle affects only the multiplication, the coradical filtration of a cocycle deformation
$A$ of $\toba(V) \# H$ remains unchanged, hence it is isomorphic to $\toba(V) \# H$ as
coalgebras. Also, it is possible to decide when  $A$ is a lifting of $\toba(V)$ over $H$.

\smallbreak

Set $\mT(V)=T(V)\# H$, $\prov=\B(V)\# H$.  Our strategy is as follows:
\begin{enumerate}\renewcommand{\theenumi}{\alph{enumi}}
\renewcommand{\labelenumi}{(\theenumi)}
\item We decompose a minimal set of generators of the ideal defining $\B(V)$
and recover $\prov$ as the last link in a chain of subsequent quotients
$\mT(V)\twoheadrightarrow\B_1\# H\twoheadrightarrow\dots\twoheadrightarrow
\B_n\# H\twoheadrightarrow \prov$.
We choose this
decomposition in such a way that every intermediate quotient is achieved by dividing by
primitive elements in $\B_i$, $i=1,\dots, n$.

\item At each step, we compute the Galois objects of $\prov_{i+1}$ as
quotients of the
Galois objects of $\prov_i$, following the results in \cite{G}. We start with the trivial
Galois object for $\mT(V)$. In the final step, we have a set $\Lambda$ of Galois objects
of $\prov$ and hence a list of cocycle deformations
$L$, which arise as $L\simeq L(\prova,\prov)$, for $\prova\in\Lambda$ as in \cite{S}.

\item We check that any lifting is obtained as one of
these deformations.
\end{enumerate}

The paper is organized as follows: In Section 2 we fix the notation and introduce the
preliminaries on Hopf algebras, Nichols algebras, cocycles and Hopf Galois objects. In
Section 3 we recall the two theorems in \cite{G} about cleft and Galois objects of
quotient Hopf algebras and study
the validity of the hypotheses of these results in order to apply them in our
context. In Section 4 we
investigate the shape of any lifting candidate of a Nichols algebra $\B(V)$. We also
study this problem in the opposite sense, that is to say we investigate the shape of the
graded algebra associated to a
deformation. Finally, in Section 5 we present our strategy to compute all cocycle
deformations of $\B(V)\# H$. As an
illustration, we apply it to classify all liftings of a Nichols algebra associated to an
example with diagonal braiding. We end this article with a question related to the extent
of this strategy.

\section{Conventions and preliminaries}

\subsection{Conventions} The base field, denoted by $\ku$, is assumed to be algebraically
closed.
Let $G$ be a group; then $Z(G)$, resp. $\widehat{G}$, denotes its center, resp. the group of multiplicative characters.
If $\chi\in \widehat{G}$ and $V$ is a $G$-module, then $V^{\chi}$ denotes the isotypic
component of $V$ of type $\chi$. Let $A$ be a $\ku$-algebra and $S\subset A$ a subset.
Then we denote by
$Z(A)$ the center of $A$, by $\langle S\rangle$ (or $\langle S\rangle_A$ if an explicit
mention to $A$ is needed) the two-sided ideal generated by $S$ and by $\ku\langle
S\rangle$ the
subalgebra generated by $S$.

Let $H$ be a Hopf algebra. We will use the (summation
free) Sweedler's notation
for the
comultiplication $\Delta$, $\varep$ will denote the counit and $\Ss$ the antipode.
 Where
needed, we stress the connection with $H$ by a subscript $H$, {\it e.g.} $\Delta_{H}$. We
denote by $H_{[0]}$ the {\it coradical} of $H$
and by $(H_{[n]})_{n\in \N}$ the
coradical filtration; $G(H)$ is the group of group-like elements. For $g,h\in G(H)$, we
denote by $\mP_{g,h}(H) =\{u\in H: \Delta(u)=u\ot
h + g\ot u\}$ the set of $(g,h)$ skew-primitive elements in $H$; $\mP(H) =\mP_{1,1}(H)$ for short. If $A$ is
a right (resp. left) $H$-comodule algebra,
then $A^{\co H}$ (resp. $\,^{\co H}A$) denotes the subalgebra of right (resp. left)
coinvariants. The right adjoint action of $H$ on itself is
\begin{align}\label{eq:right-adjoint}
\ad_r(h) (b) &= \Ss(h_{(1)})b h_{(2)},& &b,h\in H.
\end{align}

Right, resp. left, coactions are denoted by $\rho$, resp. $\lambda$. We shall also use the
Sweedler's notation for coactions.

Given a Hopf algebra $H$ with bijective antipode, we denote by $\ydh$, resp. $\hyd$, the
category of left, resp. right, Yetter-Drinfeld modules over $H$. If $K\subseteq H$ is a Hopf subalgebra, then $\mathcal{YD}^H_{K}$ is the category whose objects are $H$-comodules and $K$-modules, with the compatibility inherited from $\mathcal{YD}^H_{H}$, and
$H$-colinear, $K$-linear morphisms. We refer to \cite{Mo} for unexplained notation
and notions about Hopf algebras.

We say that $(g, \chi)$, with $g\in G(H)$ and $\chi\in \Alg(H, \ku)$,
is a \emph{YD-pair} when the following equivalent conditions hold for all $h\in H$:
\begin{align}\label{eq_yd-pair}
\chi(h)\,g = \chi(h\_{2}) h\_{1}\, g\, \Ss(h\_{3}) \Longleftrightarrow \chi(h\_{1})\,g\, h\_{2} =  h\_{1}\, g \,\chi(h\_{2}).
\end{align}
In particular, such $g$ should belong to $Z(G(H))$. If $(g, \chi)$ is a YD-pair, then $\ku_g^{\chi}$ denotes the vector space $\ku$ with coaction
$x\mapsto  g\ot x$ and action $h\cdot x=\chi(h)x$, for $x\in \ku$, $h\in H$; \eqref{eq_yd-pair} guarantees that $\ku_g^{\chi}\in \ydh$.
Conversely, any one-dimensional Yetter-Drinfeld module over $H$ arises in this way.
If $V \in \ydh$, then $V_g^\chi$ denotes the isotypic component of $V$ of type $\ku_g^{\chi}$.

\smallbreak
We refer to \cite{M,EGNO,Mug-revuma} for details about
braided Hopf algebras, that is Hopf algebras in braided tensor categories.
Recall that a Nichols algebra $\B(V)=\bigoplus_{n\geq 0} \B^n(V)$ is
a graded braided Hopf algebra in $\ydh$ generated by $V = \B^1(V)$ that coincides with the space $\mP(\B(V))$ of primitive elements in $\B(V)$.
We denote by
$\J(V)=\bigoplus_{n\geq 2} \J^n(V)$ the defining ideal of  $\toba(V)$, {\it i.e.} $\toba(V) = T(V)/ \J(V)$.
See \cite{AS2} for more details. A \emph{pre-Nichols algebra} is an intermediate
graded braided Hopf algebra between $T(V)$ and $\toba(V)$, see \cite{M4}.

\subsection{Cocycles}\label{subsect:cocycles}

Let $H$ be a Hopf algebra. A 2-cocycle $\sigma: H\ot  H \to \ku$ is a
convolution-invertible linear map $h\ot k\mapsto\sigma(h,k)$ satisfying, for $x,y,z\in
H$,
\begin{align}
\label{eqn:cociclo-dos} \sigma(x, 1) = \sigma(1, x) &= \varepsilon(x) \quad\text{and}\\
\label{eqn:cociclo-uno} \sigma(x_{(1)}, y_{(1)}) \sigma(x_{(2)} y_{(2)}, z) &=
\sigma(y_{(1)}, z_{(1)}) \sigma(x, y_{(2)}z_{(2)}).
\end{align}

Let $\sigma$ be a 2-cocycle. Then $\cdot_{\sigma}:H\ot H\rightarrow H$, given by
\begin{align}\label{eqn:cociclo-prod}
x\cdot_{\sigma}y &= \sigma(x\_{1}, y\_{1}) x\_{2} y\_{2}
\sigma^{-1}(x\_{3}, y\_{3}), \quad x, y\in H,
\end{align}
defines an associative product on the vector space $H$ with unit $1_H$. Moreover,
the collection $(H,\cdot_\sigma, 1_H, \Delta,\varepsilon, \Ss_\sigma)$ is a Hopf
algebra with antipode $\Ss_{\sigma} = f*\Ss*f^{-1}$, for $f =
\sigma\circ(\id\ot\,\Ss)\circ\Delta$. This Hopf algebra is denoted $H_\sigma$.

The group of convolution-invertible linear functionals of $H$ acts on the set $Z^2(H,\ku)$
of 2-cocycles. If $\alpha\in \Hom(H, \ku)$ is convolution-invertible, then $$
\sigma^\alpha(x, y) = \alpha(x_{(1)})\alpha(y_{(1)})\sigma(x_{(2)},
y_{(2)})\alpha^{-1}(x_{(3)}y_{(3)}), \quad \forall x,
y \in H
$$
is again a 2-cocycle and $\alpha^{-1}*\id *\alpha: H_{\sigma^\alpha}\longrightarrow
H_\sigma$ is an isomorphism of Hopf algebras. The quotient of $Z^2(H,\ku)$  under this
action is denoted $H^2 (H,\ku)$.

\begin{Rems}\label{rems:filtrations-in-deformations}
Let $H$ be a Hopf algebra and let $\sigma:H\ot H\to \ku$ be a 2-cocycle. Since the
comultiplications
of $H$ and $H_{\sigma}$ coincide, we have

\begin{enumerate}\renewcommand{\theenumi}{\alph{enumi}}
\renewcommand{\labelenumi}{(\theenumi)}
  \item The coradicals of $H$ and $H_{\sigma}$ coincide.
    \item The coradical filtrations of $H$ and $H_{\sigma}$ coincide; this is valid for
any wedge filtration ({\it e.g.} the {\it standard filtration} defined in
\cite{AC}).
  \item\label{obs:cocycle-coalgebra-item-b} If $C$ and $D$ are subcoalgebras
of $H$, then $C\cdot D = C\cdot_{\sigma} D$.
  \item Let $C$ be a subcoalgebra stable by the antipode $\Ss_H$. Let $K$ be the
subalgebra of $H$ generated by $C$ (a Hopf subalgebra indeed) and
set $\sigma'=\sigma_{|K\ot K}$. Then
  $K_{\sigma'}$ is the subalgebra of $H_{\sigma}$ generated by $C$.
\end{enumerate}
\end{Rems}

Given a 2-cocycle $\sigma:H\ot H\to \ku$ there is another way to define
an associative product on $H$:
\begin{align}\label{eqn:cociclo-prod-comodule}
x\cdot_{(\sigma)}y &= \sigma(x\_{1}, y\_{1}) x\_{2} y\_{2}
, \quad x, y\in H.
\end{align}
We denote this algebra by $H_{(\sigma)}$.  Then $\Delta:H_{(\sigma)}\to
H_{(\sigma)}\ot H$ is an algebra map and $H_{(\sigma)}$ becomes a right
$H$-comodule algebra. Moreover, $\left(H_{(\sigma)}\right)^{\co H}=\ku$.

\subsection{Galois objects}

Let $H$ be a Hopf algebra and let $A$ be a right $H$-comodule algebra with $\ku \simeq
A^{\co
H}$. Then $A$ is a (right) \emph{$H$-Galois object} if the \emph{canonical} linear map
$
\can:A\ot A\to A\ot H$, $a\ot b\mapsto ab\_0 \ot b\_{1}
$
is an isomorphism. Left $H$-Galois objects are defined
analogously. We set $\Gal (H)=\{\text{isomorphism classes of (right) $H$-Galois
objects}\}$. Let $H, L$ be Hopf algebras. An $(L,H)$-bicomodule algebra is an
\emph{$(L,H)$-biGalois object} if it is both a left $L$-Galois object and a right $H$-Galois object.

Let  $A$ be a right $H$-comodule algebra, $B=A^{\co H}$. The extension $B\subset A$
is called {\it cleft} if there exists an $H$-colinear convolution-invertible map
$\gamma:H\to A$ or, equivalently when
$A\simeq B\#_{\sigma} H$ for some 2-cocycle $\sigma:H\ot H\to B$, see \cite[Section
7]{Mo} and \cite{DT}. We may assume that $\gamma(1)=1$, in which case $\gamma$ is
called a {\it section}. The cocycle $\sigma$ is given by
\begin{align}\label{eq:cociclo-cleft}
\sigma(h,k)=\gamma(h\_1)\gamma(k\_1)\gamma^{-1}(h\_2k\_2), \quad h,k\in H.
\end{align}
If $A^{\co H}=\ku$, then $A$ is called \emph{cleft object}. Set
$$\Cleft
(H):=\{\text{isomorphism classes of $H$-cleft objects}\}\simeq H^2(H,\ku).$$

\begin{Rems}
(1) $B\subset A$ is a cleft extension if and only if $A$ is an
$H$-Galois extension and has the \emph{normal basis}
property, {\it i.e.} $A\simeq B\ot H$ as right $H$-comodules and left $B$-modules
\cite{DT}. If
$\gamma:H\to A$ is a section, then
\begin{align}\label{eqn:cleft-can-1}
 \can^{-1}(a\ot h)=a\gamma^{-1}(h\_1)\ot \gamma(h\_2), \qquad a\in A,\, h\in H.
\end{align}

\medbreak \noindent(2)  If $H$ is pointed, then any $H$-Galois extension is
cleft \cite[Remark 10]{G}.
\end{Rems}

Given a right $H$-Galois object $A$, there is a Hopf algebra
$L=L(A,H)$ attached to the pair $(A,H)$ in such a way that $A$
becomes an $(L,H)$-biGalois object \cite[Section 3]{S}. As an algebra,
$L(A,H)=(A\ot A^{\op})^{\co H}$\label{def:L(A,H)}; the coproduct $\Delta_L$ and the coaction
$\lambda:A\to L\ot A$
are:
\begin{align}
\notag \Delta_L\Big(\sum_ix_i\ot y_i \Big) &= \sum_i x_{i(0)}\ot
\can^{-1}(1\ot x_{i(1)})\ot y_i,& &\sum_i x_i\ot y_i\in L;
\\
\label{eqn:comultiplication L} \lambda(x)  &= x_{(0)} \ot \can^{-1}(1 \ot x_{(1)}), &
&x\in A.
\end{align}
Here $\can^{-1}$ is the inverse of the {\it right canonical map} $\can:A\ot A\to
A\ot H$. In turn, the inverse of the {\it left canonical map} $\can:A\ot A\to L(A,H)\ot A$
is:
\begin{align}\label{eq;can-1}
&\can^{-1}\Big((\sum_i x_i\ot y_i)\ot a \Big)= \sum_i x_i\ot y_ia, &&\sum_i x_i\ot y_i\in
L, a\in A.
\end{align}
The Hopf algebra $L$ is uniquely characterized by this property
\cite[Theorem 3.3]{S}: if $L'$ is another bialgebra and $\lambda'$ is a
left $L'$-coaction on $A$ making it an $(L',H)$-biGalois object,
then there exists a unique isomorphism $\vb:L\to L'$ such that
$\lambda'=(\vb\ot\id)\lambda$. Explicitly, see
\cite[Lemma 3.2]{S},
\begin{align}\label{eqn:f-schauenburg}
 \vb\Big(\sum_i x_i\ot y_i\Big)\ot 1_A=\sum_i\lambda'(x_i)(1\ot y_i), \qquad \sum_i
x_i\ot y_i\in L.
\end{align}

\begin{Rem}
If $\sigma\in Z^2(H,\ku)$, then $L(H_{(\sigma)},H)\simeq H_{\sigma}$ \cite[Theorem
3.9]{S}.
\end{Rem}

\section{Hopf Galois objects for quotient Hopf
algebras}

Our argument involves a recurrence on a chain of Hopf algebra quotients. We will use
\cite[Theorems 4 \& 8]{G}, which we cite next, to study cocycle deformations for a
quotient Hopf algebra.

We start with some preliminaries. Let  $\pi:L\to K$ be a projection of Hopf
algebras with bijective antipode. Then the right coideal subalgebra $X =
{^{\co K}\hspace{-2pt}L}$ of the left coinvariants is an algebra in $\ydl$ with
the right adjoint action \eqref{eq:right-adjoint} and the coaction
given by the restriction of $\Delta$.

Let $A\in \Gal(L)$. For $h\in L$, write $\sum_i \ell_i(h)\ot
r_i(h)=\can^{-1}(1\ot h)\in A\ot A$. Then $A$ is an algebra in $\ydl$ via
the {\it Miyashita-Ulbrich action} \cite{DT1}
\begin{align}\label{eq:Miyashita-Ulbrich}
a\leftharpoonup h &= \sum_i\ell_i(h)ar_i(h),& &a\in A,\, h\in L.
\end{align}   If $A,B$ are algebras in
$\ydl$, $\Alg_L^L(A,B)$ denotes the set of algebra morphisms in $\ydl$ between them.

\begin{Thm}\label{thm:gunther-teo4}\cite[Theorem 4]{G}
Let $L$, $K$, $\pi$ and $X ={^{\co K}\hspace{-2pt}L}$ be as above. Assume that $L$ is left
and right
$K$-coflat. There are bijective
correspondences
$$
\xymatrix{ {\Gal(K)} \ar@<1ex>[0,1]^-{\Phi}  & {\{(A, f): [A]\in
\Gal(L), f \in \Alg_L^L(X, A)\}}\ar[0,-1]^-{\Psi}/\sim,}
$$
\begin{align*}
&  \Psi\big([(A, f)]\big)=[A/Af(X^+)], && \Phi\big([B]\big)=\left[\left(B \Box_{K} L,
x\mapsto 1\otimes
x\right)\right].
\end{align*}
The equivalence $\sim$ is defined so that $(A,f)\sim (A',f')$ if and only if
there
exists
an isomorphism $\alpha:A\to A'$ of $L$-comodule algebras such that $f'=\alpha\circ
f$. The coaction on $A/Af(X^+)$ is
given by $(\tau\ot\pi)\lambda_A$, for $\tau:A\to A/Af(X^+)$ the projection. If
$B\in\Gal(K)$, then the $L$-coaction on $B\Box_{K} L$ is $\id_B\ot\Delta_L$.

If there is a subcoalgebra of $L$ that is mapped isomorphically onto
the coradical of $K$, then this correspondence restricts to cleft objects.
 \qed\end{Thm}

There is another approach to compute cleft objects of quotient Hopf algebras given by
an expansion of \cite[Theorem 8]{G}. To prove this, we will use the following result of
Takeuchi. Let us recall that a  right coideal subalgebra $B$ of a
Hopf algebra $\cH$  is \emph{normal} when it is stable under the right adjoint action
\eqref{eq:right-adjoint}.

\begin{Thm}\label{thm:takeuchi-correspondance}
\cite[Theorem 3.2]{T2}
Let $\cH$  be a Hopf algebra with bijective antipode. There exist mutually inverse
bijective correspondences
between the set of Hopf ideals $I$ such that $\cH$ is $\cH/I$-coflat and the set of normal
right coideal subalgebras $B$ such that $\cH$ is right $B$-faithfully flat given by
\begin{align*}
&I \text{ Hopf ideal } \rightsquigarrow {\mathcal X}(I)=\cH^{\co\cH/I}; \\
&B \text{ normal right coideal subalgebra}\rightsquigarrow {\mathcal I}(B)=\cH B^+.
\qed
\end{align*}
\end{Thm}

If $A, A'$ are right $L$-comodule algebras, then $\Alg^L(A,A')$ is the set of comodule
algebra morphisms between them. If $X\subset L$ is a right coideal subalgebra, then
$N(X)$\label{def:NX}
is the subalgebra generated by $\{\Ss(h\_1)xh\_2:\,h\in L,
x\in X\}$; this is the normal subalgebra generated by $X$.

\begin{Thm}\label{thm:gunther-extendido}
Let $L$ be a Hopf algebra with bijective antipode. Let $Y\subset L$ be a right coideal
subalgebra. Set $I=LY^+L$ and $K=L/I$; then $K$ is a quotient Hopf algebra of $L$. Assume
that $L$ is $K$-coflat and that $L$ is faithfully flat over $N(Y)$. Then there are
bijective
correspondences
\begin{equation}\label{eq:gunther8}
\xymatrix{ {\Cleft(K)} \ar@<1ex>[0,1]^-{\Phi}  & {\Big\{ (A, f): \begin{array}{l}
                                                                  [A]\in
\Cleft(L), f \in \Alg^L(Y, A) \\ \text{ such that }  Af(Y^+)A\neq A
                                                                 \end{array}
\Big\}}\ar[0,-1]^-{\Psi}/\sim,}
\end{equation}
\begin{align*}
 &\Psi\big([(A, f)]\big)=[A/Af(Y^+)A], && \Phi\big([B]\big)=\left[\left(B \Box_{K} L,
x\mapsto 1\otimes
x\right)\right].
\end{align*}
The corresponding coactions and the
relation $\sim$ are as in Theorem \ref{thm:gunther-teo4}.
\end{Thm}
\pf
Set $X=N(Y)$. First, we use Theorem \ref{thm:takeuchi-correspondance} to show $X=\,^{\co
K}L$.
Indeed, we have, on the one hand, ${\mathcal I}(X)=LN(Y)^+=LY^+L=I$. On the other,
${\mathcal X}(I)=\,^{\co K}L$ and the statement follows.

The proof now runs as that of \cite[Theorem 8]{G}, as the hypotheses $\,^{\co K}L \cap
L_{[0]}\subseteq N(Y)$ (which we recover trivially) and that of $L$ being pointed in {\it
loc. cit.} are precisely used  to show $N(Y)=\,^{\co K}L$.
\epf

In order to apply Theorems \ref{thm:gunther-teo4} and \ref{thm:gunther-extendido} we need
to  investigate when a Hopf
algebra $L$ is coflat over a quotient Hopf algebra $K$. This will be the content of
Subsection \ref{subsect:quotients-fflat}. For Theorem \ref{thm:gunther-extendido}, we need
to study when $L$ is faithfully flat over a right coideal
subalgebra, we will also deal with this question in the next subsection. Also, see Section
\ref{sec:question} for a detailed comparison of
these two theorems.

\subsection{On the coflatness of quotients}\label{subsect:quotients-fflat}
Fix a Hopf algebra $H$ with bijective antipode. Let $R$ be a connected ({\it i.e.}
the coradical of $R$ is $\ku$) Hopf algebra in
${}^H_H\mathcal{YD}$. In particular, the antipode of $R$, hence that of $A=R\# H$, is
bijective. Clearly, see \cite[Section 1]{M3}, we have

\begin{itemize}
\medbreak   \item If $B$ is a right coideal subalgebra of $R$, then $R/RB^+$ is a
quotient left $R$-module coalgebra of $R$.

\medbreak \item
If $T$ is a quotient left $R$-module coalgebra of $R$, then the left $T$-coinvariants
${}^{\co T}R$ form a right coideal subalgebra of $R$.
\end{itemize}

Recall that the structures on $R$ arise from the obvious Hopf algebra maps $H \to A \to
H$, whose composite
is the identity on $H$, as follows: we have $R= A^{co\, H}$, so that $R$ is a left coideal
subalgebra of $A$,
and is thus an algebra and left $H$-comodule, while we have $R = A/AH^{+}$, so that
$R$ is a quotient left $A$-module coalgebra of $A$,
and is thus a coalgebra and left $H$-module; the last left $H$-module structure coincides
with the adjoint action.
Let $P$ (resp. $Q$) denote the braided Hopf algebra in
${}^{H^{\op}}_{H^{\op}}\mathcal{YD}$
(resp. ${}^{H^{\cop}}_{H^{\cop}}\mathcal{YD}$) which arises from $H^{\op} \to A^{\op} \to
H^{\op}$ (resp.
$H^{\cop} \to A^{\cop} \to H^{\cop}$). Then $P = Q = R$ as vector spaces. As an algebra,
$P$ equals the opposite
algebra $R^{\op}$ of $R$, while as a coalgebra, $Q$ equals the co-opposite coalgebra
$R^{\cop}$ of $R$.

\begin{lema}\label{lema:yd-subobjects} (i) The sub-objects (resp.,  right coideal
subalgebras)
of $R$ in ${}^H_H\mathcal{YD}$
coincide with those of $P$ in
${}^{H^{\op}}_{H^{\op}}\mathcal{YD}$.

(ii) The quotient objects (resp., left module coalgebras) of $R$ in ${}^H_H\mathcal{YD}$
coincide with those of $Q$ in ${}^{H^{\cop}}_{H^{\cop}}\mathcal{YD}$.
\end{lema}

\begin{proof}
(i) Since the comultiplication does not change, $R = P$ as left comodules over the
coalgebra
$H = H^{\op}$. If $h = \Ss_H^{-1}(k) \in H$, $a \in A$, then the adjoint actions of $H$
and $H^{\op}$ are related by
$$\ad_H h (a)  =  h\_{1}a\Ss_H(h\_{2})=  k\_{1}\cdot_{\op}
a\cdot_{\op}\Ss_{H^{\op}}(k\_{2}) = \ad_{H^{\op}} k (a). $$
This settles the claim for sub-objects. Let now $X$ be a sub-object of $R$, or
equivalently of $P$. Clearly,
$X$ is a subalgebra of $R$ if and only if it is a subalgebra of $P = R^{\op}$.
For $x\in R$, let $x \mapsto \sum\, x^{(1)} \otimes x^{(2)}$ denote the coproduct on $R$.
Then
the coproduct $\Delta(x)$ on $A$ is given by
\begin{align*}
 \Delta(x) &=
 (x^{(2)})\_{-1}(\Ss_H^{-1}((x^{(2)})_{(-2)})\rightharpoonup x^{(1)}) \otimes
(x^{(2)})\_{0}.
\end{align*}
Hence $\Delta_P$ is given by
$x \mapsto  \Ss_H^{-1}((x^{(2)})\_{-1})\rightharpoonup x^{(1)} \otimes (x^{(2)})\_{0}$.
It follows that $X$ is a right coideal of $R$, if and only if it is such of $P$. (ii)
is similar.
\end{proof}

Let $B$ be a right coideal subalgebra of $R$ in ${}^H_H\mathcal{YD}$. Then one can define
the category
$({}^H_H\mathcal{YD})^R_B$ of right $(R, B)$-Hopf modules in ${}^H_H\mathcal{YD}$; a
\emph{right $(R, B)$-Hopf module} is here
a right $B$-module and right $R$-comodule in ${}^H_H\mathcal{YD}$ which satisfies the
compatibility condition
formulated as in the ordinary situation, but involving the braiding
$R \otimes B \overset{\simeq}{\longrightarrow} B \otimes R$.

\begin{lema}\label{lem:claim}
 Every object $M$ in $({}^H_H\mathcal{YD})^R_B$ includes a sub-object $X$ in
${}^H_H\mathcal{YD}$ such that the
action map $X \otimes B \to M$ is a bijection, necessarily an isomorphism in
$({}^H_H\mathcal{YD})_B$.
\end{lema}
\pf
The lemma follows from the following claim.
\begin{claim} If $M \ne 0$, then $M$ includes
a non-zero sub-object $N$ in $({}^H_H\mathcal{YD})^R_B$ which includes a sub-object $X$
in ${}^H_H\mathcal{YD}$ such that $X \otimes B \overset{\simeq}{\longrightarrow} N$.
\end{claim}
Indeed, assume that we have proven the claim. We consider
all pairs $(N, X)$, where $N$ is a sub-object of $M$ in $({}^H_H\mathcal{YD})^R_B$, and
$X$ is a sub-object of $N$
in ${}^H_H\mathcal{YD}$ such that $X \otimes B \overset{\simeq}{\longrightarrow} N$
naturally, and
introduce the
natural order given by inclusion to the pairs. By Zorn's Lemma we have a maximal pair $(N,
X)$.
To see $N = M$, suppose on the contrary $N \subsetneq M$. With the assumed fact applied to
$M/L$,
we have
sub-objects  $\tilde{N} \subset M$ in $({}^H_H\mathcal{YD})^R_B$ and $\tilde{X} \subset
\tilde{N}$ in
${}^H_H\mathcal{YD}$ such that $N \subsetneq \tilde{N}$, $X \subsetneq \tilde{X}$ and
$\tilde{X}/X \otimes B \overset{\simeq}{\longrightarrow} \tilde{N}/N$.
A map of short exact exact sequences which is isomorphic on the kernels and the cokernels
shows that $\tilde{X} \otimes B \overset{\simeq}{\longrightarrow} \tilde{N}$, which
contradicts the maximality
of $(N, X)$, and hence shows $N = M$. This argument is the same as the one in
\cite[Proposition 1]{R}.

\smallbreak
We now prove the claim. Suppose $0 \ne M \in ({}^H_H\mathcal{YD})^R_B$. We wish to prove
$M$ includes a nonzero pair.
Set $X = M^{\co R}$. This is a subject of $M$ in ${}^H_H\mathcal{YD}$, and is the socle
$\mathrm{soc}\, M$
of the right $R$-comodule $M$, whence $X \ne 0$.
The tensor product $X \otimes B$ is naturally an object in $({}^H_H\mathcal{YD})^R_B$
whose $R$-comodule socle $\mathrm{soc}(X \otimes B) = X$. We see that
$f : X \otimes B \to M$, $f(v \otimes b) = vb$
is a morphism in $({}^H_H\mathcal{YD})^R_B$, which is injective since it is restricted to
the identity
on the socles. If $L = \mathrm{Im}\, f = XB$ then $(N, X)$ is a desired pair.
\epf

\begin{prop}\label{pro:masuoka}
Let $R$ be a connected Hopf algebra in ${}^H_H\mathcal{YD}$.
\begin{enumerate}\renewcommand{\theenumi}{\alph{enumi}}
\renewcommand{\labelenumi}{(\theenumi)}
\item\label{i}
$R$ is a free left and right module over every right coideal subalgebra.

\smallbreak\item\label{ii}
$R$ is a cofree left and right comodule over every quotient left module coalgebra $T$.

\smallbreak\item\label{iii}
$B \mapsto R/RB^+$ and $T \mapsto {}^{\co T}R$ give a bijection between the
set of right coideal subalgebras $B$ of $R$ and the set of quotient left
$R$-module coalgebras $T$ of $R$.

\smallbreak\item\label{iv}
If $B$ and $T$ correspond to each other via this bijection, then
there exists a
left $T$-colinear and right $B$-linear isomorphism $T \otimes B
\overset{\simeq}{\longrightarrow} R$.
\end{enumerate}
\end{prop}

\begin{proof} \eqref{i}
 Let $B$ be a right coideal subalgebra. First, we prove the right $B$-freeness.
Notice that $R \in
({}^H_H\mathcal{YD})^R_B$. The right $B$-freeness in \eqref{i} follows from Lemma
\ref{lem:claim}. By Lemma \ref{lema:yd-subobjects} (i), the just proved result applied to
the $P$ of the lemma
implies the left $B$-freeness\footnote{ One can define the analogous category
${}_B({}^H_H\mathcal{YD})^R$. But, it is
impossible
to discuss as above, because for $M \in {}_B({}^H_H\mathcal{YD})^R$, the right
$R$-comodule $B \otimes M^{\co R}$
is not isomorphic to a direct sum of copies of $B$, and so $\mathrm{soc}(B \otimes
M^{\co R}) = M^{\co R}$
may not be true.}.

\eqref{iii} Let $T$ be as in \eqref{ii}, and set $B = {}^{\co T}R$. We see that the
natural
left $T$-comodule structure $R \to T \otimes R$
on $R$ is right $B$-linear. The base extension along $B \to R$
induces a right $B$-linear and left $T$-colinear map
$ g : R \otimes_B R \to T \otimes R.$
This is induced from the canonical isomorphism $R \otimes R
\overset{\simeq}{\longrightarrow} R \otimes R$,
and hence is a surjection. Note that $T$ is also connected.

Since $R$ is left $B$-free as was shown in \eqref{i}, the left $T$-comodule socle of $R
\otimes_B R$ equals $B \otimes_B R$.
It follows that $g$ is injective, and hence bijective, since it restricts to
$\mathrm{id}_R$ on the
left $T$-comodule socles.

The bijection $g$ together with the left $B$-freeness of $R$ shows that $R$ is an
injective
cogenerator (or equivalently, faithfully coflat) as a left $T$-comodule; see also
\cite[Proposition 1.4(1)]{M3}. The desired one-to-one correspondence follows
just as in the ordinary situation; see \cite[Proposition 1.4(2)]{M3}.

\eqref{iv} Let $B$ and $T$ correspond to each other.
The left $T$-injectivity of $R$ allows the inclusion $k \to R$ of left
$T$-comodules to extend to a unit-preserving left $T$-colinear map $h : T \to R$. The
right
$B$-linearization of $h$
$$ h_B : T \otimes B \to R, \ h_B(t \otimes b) = h(t)b $$
is right $B$-linear and left $T$-colinear and injective, since it restricts to
$\mathrm{id}_B$ on the $T$-comodule socles. It is an isomorphism, since $T
\otimes B$ is
$T$-injective.

\eqref{ii} Let $T$ be as in \eqref{ii}. We see from Part \eqref{iv} that $R$ is
left $T$-cofree. Lemma \ref{lema:yd-subobjects} (ii) shows that $R$ is also right
$T$-cofree.
\end{proof}

\begin{Cor}\label{cor:coflat}
Let $H$ be a cosemisimple Hopf algebra, let $R, T$ be connected braided Hopf algebras
in $\ydh$, such that $T$ is a quotient of $R$. Then $R\# H$ is left and right cofree over
$T\# H$. In particular, it is left and right coflat.
\end{Cor}
\pf
As $H$ is cosemisimple, the coalgebra surjection $\id\ot\,\varep:T\# H\to T$ is a
cosemisimple coextension, that is a left or right $T\# H$-comodule is injective if it is
injective as a $T$-comodule. Since $R$ is left $T$-cofree as a $T$-comodule and so left
$T$-injective, it
follows that $R\# H$, being left $T$-injective, is left $T\#H$ injective. Note that the
coradical of $T\#H$ is isomorphically liftable to the coradical of $R\#H$, since both of
them coincide with $H$. It follows by the left version of \cite[Theorem 4.2]{Sc} that
there is a left $T\#H$-colinear and
right $\,^{\co T\#H}(R\# H)$-linear isomorphism
$$
R\#H\simeq T\#H\ot \,^{\co T\#H}(R\# H).
$$
By switching the sides one can present $R \# H$, $T \# H$ as smash products $H \# R'$, $H
\# T'$
of braided Hopf algebras $R'$, $T'$ in $\mathcal{YD}^H_H$, such that $T'$ is a quotient of
$R'$, and prove that
$R'$ is right (and left) $T'$-injective, which shows as above that there is a right $T\#
H$-colinear
and left $(R\# H)^{\co\, T\# H}$-linear isomorphism $R\# H \simeq (R\# H)^{\co\, T\#
H}\otimes(T\# H)$, by \cite[Theorem 4.2]{Sc}.
\epf

\begin{Cor}\label{cor:flat}
Let $R$ be a connected braided Hopf algebra
in $\ydh$ and let $B\subseteq R$ be a left coideal subalgebra. Let $L=R\#H$, then $B\#1$
is a left coideal subalgebra of $L$ and $L$ is a left $B\#1$-free module.

If $Y=\Ss(B\#1)$, then $L$ is a right $Y$-free module. In particular,
it is right $Y$-faithfully flat.
\end{Cor}
\pf
The formula for $\Delta_{R\#H}$ shows that $B\#1$ is a left coideal
subalgebra. Now, the first statement  follows since $L$ is left $R$-free and $R$ is left
$B$-free by Proposition \ref{pro:masuoka} \eqref{i}. The second is straightforward.
\epf

\section{The shape of all possible deformations}\label{sect:shape}

Let $H$ be as in \eqref{eqn:H} and $A$ be a Hopf algebra whose coradical is isomorphic to $H$. In the first part of this section, we assume that $H$ is also semisimple ({\it e. ~g.} when the characteristic of $\ku$ is $0$) and analyze the structure of $A$.

A fundamental information is that there exists a coalgebra $H$-bimodule projection $\Pi:A\rightarrow H$ such that $\Pi_{|H}=\id_H$ \cite[Theorem 5.9.c)]{ardizzonimstefan}. Hence $A$ is a Hopf bimodule
coalgebra over $H$ via the left and right multiplication and the coactions
$\rho_L=(\Pi\ot\id)\Delta$ and $\rho_R=(\id \ot \Pi)\Delta$. Let $P_0=0$, $P_1=\{x\in
A:\Delta(x)=\rho_L(x)+\rho_R(x)\}$ and
\begin{align*}
P_n&=\{x\in A:\Delta(x)-\rho_L(x)-\rho_R(x)\in\sum_{i=1}^{n-1}P_i\ot P_{n-i}\}.
\end{align*}
Then $P_n=A_{[n]}\cap\ker\Pi$ \cite[Lemma 1.1]{andrunatale}. Clearly, $P_n$ is a Hopf
sub-bimodule of $A_{[n]}$ and $A_{[n]}/A_{[n-1]}=P_{n}/P_{n-1}$.

The canonical projection
$\pi_n:A_{[n]}\rightarrow A_{[n]}/A_{[n-1]}$ is a Hopf bimodule map and it has a section
$\iota_n$ since $H$ is semisimple and cosemisimple. Therefore
$$A\simeq H\oplus\bigoplus_{n\geq1}\iota_n(P_n/P_{n-1})$$
as Hopf bimodule. We extend $\pi_n$ to be $0$
in $\bigoplus_{m>n}\iota_m(P_m/P_{m-1})$, $n>0$, and set $\pi_{0}=\Pi$. We shall
generally
omit $\iota_m$.

\smallbreak

We recall the structure of $\gr A$. As vector spaces, $\gr
A(n)=A_{[n]}/A_{[n-1]}=P_n/P_{n-1}$. The multiplication and comultiplication of $\gr A$
are
\begin{align*}
\pi_n(x)\pi_m(y)&=\pi_{n+m}(xy)\quad \mbox{ and}\\
\Delta_{\gr A}(\pi_n(x))&=\sum_{i=0}^n\pi_i(x\_{1})\ot\pi_{n-i}(x\_{2}),\quad x\in
A_{[n]},y\in A_{[m]}.
\end{align*}
By abuse of notation, $\pi_0$ denotes the projection $\gr A\twoheadrightarrow H$ with
kernel $\oplus_{n>0}\gr A(n)$ and $\pi_{0|H}=\id$. Then $\gr A$ is a Hopf bimodule over
$H$ via the left and right multiplication and the coactions $(\pi_0\ot\id)\Delta_{\gr A}$,
$(\id \ot \pi_0)\Delta_{\gr A}$.

\medbreak

It is well-known that $\gr A\simeq(\gr A)^{\co H}\#H$ as Hopf algebras. In
\cite[Theorem 5.23]{ardizzonimstefan}, it is shown that $A^{\co H}$ is a coalgebra in
$\ydh$ such that $A\simeq A^{\co H}\#H$ as coalgebras and the multiplication in $A$ is
recovered with an extra structure on $A^{\co H}$, see also \cite{S2}.

\begin{lema}\label{le:hopfbimod structure of A}
$\pi_n:\iota_n(P_n/P_{n-1})\rightarrow\gr A(n)$ is an isomorphism of Hopf bimodules
over $H$ for all $n$.
Therefore $A^{\co H}\simeq(\gr A)^{\co H}$ in $\ydh$.
\end{lema}

\begin{proof}
If $x\in \iota_n(P_n/P_{n-1})$ and $h\in H$, then $\pi_n(hx)=\pi_0(h)\pi_n(x)$ and
$\pi_n(xh)=\pi_n(x)\pi_0(h)$. Also,
\begin{align*}
(\id \ot
\pi_0)\Delta^{\gr}(\pi_n(x))&=\sum_{i=0}^n\pi_i(x\_{1})\ot\pi_0\circ\pi_{n-i}(x\_{2}
)=\pi_n(x\_{1})\ot\pi_0(x\_{2})\\
&=(\pi_n\ot\id)(\id \ot \Pi)\Delta(x)=(\pi_n\ot\id)\rho(x).
\end{align*}
Analogously, $\pi_n$ is a left comodule map. The last assertion is easy.
\end{proof}

\begin{Rem}
Assume that the dimension of $A$ is finite. If the isomorphism $A^{\co H}\simeq(\gr A)^{\co H}$ in Lemma
\ref{le:hopfbimod structure of A} is also of coalgebras, then
$A\simeq\gr A$ as coalgebras; thus $A$ is a cocycle deformation of $\gr A$ by
\cite[Corollary 5.9]{S}.
\end{Rem}

We fix $V\in\ydh$ with $\dim V < \infty$. Let $\B(V)=T(V)/\J(V)$ be the Nichols algebra
of $V$ see \cite{AS2} and set $\T(V)=T(V)\#H$.

\begin{Def}\label{eq:properties of A and phi}
A {\it lifting map} is an epimorphism $\phi:\T(V)\rightarrow A$ of Hopf algebras such that
$\phi_{|H}=\id_H$ and $\phi_{|V\#H}:V\#H\rightarrow P_1$ is an isomorphism of Hopf
bimodules over $H$.
\end{Def}

\begin{prop}\cite[Proposition 2.4]{AV}
Let $A$ be a Hopf algebra whose coradical is a Hopf subalgebra isomorphic to $H$. Then
$A$ is a lifting of $\B(V)$ over $H$ if and only if there exists a lifting map
$\phi:\T(V)\rightarrow A$. \qed
\end{prop}

The case $H=\ku\Gamma$, $\Gamma$ an abelian group, in the above proposition has been
previously considered, see for instance \cite{Kh, He}. 

\smallskip

From now on, {\it we assume that $A$ is a lifting of $\B(V)$ over $H$ with lifting map
$\phi:\T(V)\rightarrow A$.} In particular, $V$ is a submodule of $A$ in $\ydh$.

Let $\BB^n_\J$ be a basis of $\J^n(V)$ and extend it to a
basis $\BB^n\cup\BB^n_\J$ of $V^{\ot n}$. We still denote by $\BB^n$ the basis of the
quotient $\B^n(V)$.
Then $\BB=\bigcup_n\BB^n$ is a basis of
$\B(V)$, $\BB_\J=\bigcup_n\BB^n_\J$ is a basis of $\J(V)$. Let $\BB_H$ be a basis of $H$.

\begin{Rems}\label{rem:phi, bases de BV y A} By Lemma \ref{le:hopfbimod structure of A} we
have that:
\begin{enumerate}\renewcommand{\theenumi}{\alph{enumi}}
\renewcommand{\labelenumi}{(\theenumi)}
\item\label{item:base de J rem:phi, bases de BV y A}
$\{\phi(x)-\Pi(\phi(x)):\, x\in\BB^n_\J\}\subset P_{n-1}$.
\smallbreak
\item $\{\phi(x)h-\Pi(\phi(x))h:x\in\BB^i, h\in\BB_H, 0<i\leq n\}$ is a basis of $P_n$.
\smallbreak
\item $\phi(\BB^2)H=\iota_2(P_2/P_{1})$ and $A_{[2]}\simeq(\B(V)\#H)_{[2]}$ as coalgebras.
\smallbreak
\item $\{\phi(x)h:x\in\BB,h\in\BB_H\}$ is a basis of $A$. Let $\iota:A\rightarrow\T(V)$ be
the linear map identifying this basis of $A$ with $\BB\#H$.
\end{enumerate}
\end{Rems}

The shape of the liftings is given by the following proposition. If $M\subset T(V)$ is a Yetter-Drinfeld submodule, we define the ideal
$$
\cI_M=\langle m-\iota\phi(m)\, |\, m\in M\rangle.
$$

\begin{prop}\label{prop:shape of A}
Let $M\subset T(V)$ be a Yetter-Drinfeld submodule which generates $\J(V)$. If $\cI_M$ is a Hopf ideal, then $A=\T(V)/\cI_M$.
\end{prop}
\begin{proof} Let $A':=\T(V)/\cI_M$; since $\cI_M$ is contained in the kernel of the lifting map $\phi$, we have an
epimorphism $A' \twoheadrightarrow A$ and $\cI_M\cap(\ku\oplus V)\#H=0$. Hence the coradical of $A'$ is $H$ by
\cite[Corollary 5.3.5]{Mo}. Then $\gr(A')\simeq R\#H$ where
$R\simeq T(V)/J$ for a braided Hopf ideal $J\subseteq\J(V)$. Clearly $M\subset J$,
cf.
Remark \ref{rem:phi, bases de BV y A} \eqref{item:base de J rem:phi, bases de BV y A},
then $J=\J(V)$ and $\dim(A'_{[n]}/A'_{[n-1]})=\dim(A_{[n]}/A_{[n-1]})$ for all $n\in\N$,
hence
the proposition follows.
\end{proof}

If there are no ambiguities, we identify $(\ku\oplus V)\#H$ with
its image by $\phi$ omitting the map $\iota$. We explore a case where the hypothesis of
Proposition \ref{prop:shape of A} is satisfied.

\begin{definition} A submodule $M$ of
$T(V)$ in $\ydh$ is {\it compatible with} $\phi$ when
$$\Delta(\phi(m))=\phi(m)\ot1+m\_{-1}\ot\phi(m\_{0})\mbox{ for all }m\in M.$$
\end{definition}

\medbreak

Assume $M\subset T(V)$ is compatible with $\phi$. For  $m\in M$, we may see $\phi(m)$ as
an element of $(\ku\oplus V)\#H$. Fix a basis $\{m_i\}_{1\le i \le r}$ of $M$ and let
$\{c_{ij}\}_{i,j}\subset H$ be the set of {\it comatrix elements}
associated to $M$ and $\{m_i\}_{1\le i \le r}$, that is
\begin{align}\label{eqn:comatrix}
(m_i)\_{-1}\ot (m_i)\_{0}=\sum_{j}c_{ij}\ot m_j, \quad  1\leq i\leq r.
\end{align}
If $M$ is simple, then the set $\{c_{ij}\}_{i,j}$ is linearly independent and spans
a simple coalgebra. Next lemma helps us to describe the
image $\phi(M)$.

\begin{lema}\label{lema:liftings gral}
Let $M\subset T(V)$ be compatible with $\phi$. Then
\begin{enumerate}\renewcommand{\theenumi}{\alph{enumi}}\renewcommand{\labelenumi}{
(\theenumi)}
\item\label{item:pVH phiM} $\pi_1\circ\phi_{|M}:M\rightarrow V$ is a morphism in
$\ydh$.
\smallbreak
\item \label{cor:liftings Mtimes} Assume that $M$ is simple and
$V\simeq M^m\oplus P$ with $m$ maximum. Then there exist $\lambda_1, \dots,
\lambda_m\in\ku$ such that
$$\pi_1\circ\phi_{|M}\simeq\lambda_1\id_M\oplus\cdots\oplus\lambda_m\id_M\oplus\,0.$$
\medbreak
\item\label{item:pH phiM} For $\{m_i\}_{1\le i \le r}$, $\{c_{ij}\}_{i,j}$ as
in \eqref{eqn:comatrix} there exist $a_1, \dots, a_r\in\ku$ such that
$$(\pi_0\circ\phi)(m_i)=a_i-\sum_{j=1}^ra_{j}c_{ij} \qquad i=1,\dots, r.$$
\smallbreak
\item\label{item:iso between As} Let $\Theta:A\rightarrow A'$ be an isomorphism of Hopf
algebras and let $\phi':\T(V)\rightarrow A'$ be a lifting map. If there is no $v\in V$
such
that $h\cdot v=\cou(h)v$ for all $h\in H$, then $\Theta\phi (V)=\phi'(V)$.
\smallbreak
\end{enumerate}
\end{lema}

\begin{proof}
\eqref{item:pVH phiM} Clearly $\phi(M)\subset A_{[1]}$. Since $\pi_1\circ\phi_{|M}$ is a
morphism of bicomodules over $H$ by Lemma \ref{le:hopfbimod structure of
A}, $(\pi_1\circ\phi)(M)\subset V$.
\eqref{cor:liftings Mtimes} is a particular case of \eqref{item:pVH phiM}.
We prove \eqref{item:pH phiM}. Recall that $\ydh$ is a semisimple category, so 
$M=\bigoplus_{l=1}^n M_l$ where each $M_l$ is a simple
$H$-comodule. If $M$ is a
simple $H$-comodule, then
\eqref{item:pH phiM} follows from \cite[Lemma 2.1]{AV} since
$$\Delta((\pi_0\circ\phi)(m_i))=(\pi_0\circ\phi)(m_i)\ot1+\sum_{j}c_{ij}
\ot(\pi_0\circ\phi)(m_j).$$
Otherwise the same argument can be applied on each simple summand $M_l$. 

\eqref{item:iso between As} If we consider $A_{[1]}$ as a right $H$-comodule via the
projection $(\cou\#1)\circ\Theta$, then $\Theta\phi (V)\subset(\ku\oplus\phi'(V))\#1$. Now
if we consider $A_{[1]}'$ as a left $H$-module via $\ad\circ\Theta$, then \eqref{item:iso
between As} follows by hypothesis.
\end{proof}

Lemma \ref{lema:liftings gral} has been refined in the copointed case, {\it i.e.} when
$H$ is the function algebra on a
finite group, in \cite[Lemma 3.1]{GIV}.

\begin{lema}\label{lem:comp-with-phi} Let $M,N\subset T(V)$ be Yetter-Drinfeld submodules.
\begin{itemize}
\item[(a)] If $M$ is included in the homogeneous component of minimum degree of $\J(V)$, then $M$ is compatible with $\phi$.
\item[(b)] Assume that $M$ is compatible with $\phi$, $\cI_M$ is a Hopf ideal and
\begin{align}\label{eq:hypothesis generadors of ker phi}
\Delta(n)-n\ot1-n\_{-1}\ot n\_{0}\in \cI_M\ot \mT(V) + \mT(V)\ot \cI_M
\end{align}
holds for every $n\in N$. Then $N$ is compatible with $\phi$ and
$\cI_{M\oplus N}$ is a Hopf ideal.
\end{itemize}
\end{lema}
\pf
(a) By hypothesis, $M\subset\mP(T(V))$ and then $M$ is compatible with $\phi$.

(b) Since $\cI_M\subset\ker\phi$, $N$ is compatible with $\phi$ by
\eqref{eq:hypothesis generadors of ker phi}. Moreover, applying Lemma \ref{lema:liftings gral} \eqref{item:pVH
phiM} and \eqref{item:pH phiM}, we see that $\langle m-\phi(m)\rangle_{m\in N}$
is a Hopf ideal in $\T(V)/\cI_M$ by \eqref{eq:hypothesis generadors of ker phi}. Hence $\cI_{M\oplus N}$ is a Hopf
ideal of $\T(V)$.
\epf

\begin{definition} A \emph{good module of relations} is a graded Yetter-Drinfeld submodule
$M=\bigoplus_{i=1}^{t}M^{n_i}\subset T(V)$ where
$M^{n_i}\subseteq T^{n_i}(V)$,  with $M^{n_i} \neq 0$ and
$n_i<n_{i+1}$ for all $i$, which generates $\J(V)$ such that the Yetter-Drinfeld
submodules $\bigoplus_{i=1}^{s} M^{n_i}, M^{n_{s+1}} \subset T(V)$ satisfy
\eqref{eq:hypothesis generadors of ker phi} for all $s=1, \dots, t-1$.
\end{definition}

Now we describe the liftings of $\B(V)$ over $H$ when $\J(V)$ is generated by a good module of relations.

\begin{Thm}\label{cor:generadors of ker phi} Let $A$ be a lifting of $\B(V)$ over $H$,
with lifting map $\phi$.
Let $M$ be a good module of relations for $\B(V)$. Then $A\simeq\mT(V)/\cI_M$.
\end{Thm}
\pf
Follows from Proposition \ref{prop:shape of A} and Lemma \ref{lem:comp-with-phi}.
\epf

Theorem \ref{cor:generadors of ker phi} characterizes the liftings in
the case
in which the relations are deformed by elements in the first term of the coradical
filtration. This is the case in \cite{AS3, AG2, AV, GGI, FG} (actually in those papers, the relations are deformed by elements in the zeroth term of the coradical
filtration). However, there exist
examples in which this does not hold, see Example \ref{exa:f5} below, also \cite{He,
GIV}.

\begin{Ex}\cite[Theorem 5.4]{GIV}\label{exa:f5}
Set $\ku=\CC$ and let $\bF_5$ denote the finite field of 5 elements. Consider the {\it
affine rack}
$X=(\bF_5,2)$ and the constant cocycle
$q\equiv-1$. The Nichols algebra $\B(X,q)$, computed
in \cite{AG1}, has dimension $1280$ and can be presented by generators $x_0, \dots,
x_4$ and relations
 \begin{align}\label{eq:rel of f5}
\begin{split}
&x_i^2, \qquad x_ix_j+x_{2j-i}x_i+x_{3i-2j}x_{2j-i}+x_jx_{3i-2j}
\qquad 0\leq i,j\leq
 4,\\
&x_1x_0x_1x_0 + x_0x_1x_0x_1.
\end{split}
\end{align}
Let $C_8$ be the cyclic group of order 8 and let $t$ denote a generator.
Consider $C_8$ acting on $\Z_5$ by $t\cdot i=2i$, $i\in\Z_5$, and set
$\Gamma=\Z_5\rtimes_2 C_8$. Then $\B(X,q)$
can be realized in $\ydgg$. Let $\prov=\B(X,q)\#\ku \Gamma$.
Set $g_i=i\times t\in\Gamma$, $i\in\Z_5$. Let $V$ be the linear span of $\{x_0,\dots,
x_4\}$. If $L$ is a deformation of $\toba(X,q)$, then there
exist scalars $\lambda_1,\lambda_2,\lambda_3\in\ku$ such that $L$ is the
quotient of $T(V)\#\ku \Gamma$ by the ideal generated by
\begin{align*}
&x_0^2 - \lambda_1 (1- g_0^2),  \qquad x_0x_1 + x_2x_0 + x_3x_2
+ x_1x_3 - \lambda_2(1 -
g_0g_1), \\
&x_1x_0x_1x_0 + x_0x_1x_0x_1 -s_X - \lambda_3(1-g_0^2g_1g_2),
\end{align*}
for $s_X=\lambda_2\, (x_1x_0 +
x_0x_1)+\lambda_1\, g_1^2(x_3x_0+ x_2x_3) - \lambda_1\, g_0^2(x_2x_4+ x_1x_2) +
\lambda_2\lambda_1\,g_0^2(1- g_1g_2)\in L_{[2]}$. Hence, the relation $x_1x_0x_1x_0
+ x_0x_1x_0x_1$ of $\B(X,q)$ is not deformed by elements in the first term of the
coradical filtration.
\end{Ex}

\subsection{The shape of the Hopf algebra $L(A, \prov)$} Let $H$, $V$ be as in \eqref{eqn:H}, \eqref{eqn:V} and $\B$ be a
pre-Nichols algebra over $V$. We set $\prov=\B\# H$ and let $\pi:\mT(V)\to \prov$ be
the canonical projection. Consider $\mT(V)$ as an $\prov$-comodule
algebra via $(\id\ot\pi)\Delta_{\mT(V)}$.  Let $A\in\Gal(\prov)$ be provided with a
projection $\tau:\mT(V)\to A$ which is a morphism of comodule algebras.

We shall need explicit presentations of the Hopf algebra $L(A, \prov)$.
\begin{prop}\label{pro:projection}
Let $\wp=(\tau\ot\tau)(\id\ot
\Ss)\Delta_{\mT(V)}\in\Alg(\mT(V), A\ot
A^{\op})$. Then $L(A,\prov) = \wp(\mT(V))$.
\end{prop}
\pf
On one hand, observe that $(\id\ot\Ss)\Delta_{\mT(V)}:\mT(V)\to L(\mT(V),\mT(V))$ is a Hopf algebra isomorphism. On the other hand, there is a
Hopf algebra map $L(\mT(V),\mT(V)) \to L(A,\prov)$ making the following diagram commutative:
\begin{equation*}
\xymatrix{L(\mT(V),\mT(V)) \ar@{^{(}->}[rr] \ar[d] && \mT(V)\ot
\mT(V)\ar[d]^{\tau\ot\tau}  \\
L(A,\prov) \ar@{^{(}->}[rr]   && A \ot A, }
\end{equation*}
which gives rise to
\begin{equation*}
\xymatrix{L(\mT(V),\mT(V))\ot\mT(V) \ar[rr]^{\can^{-1}} \ar[d] && \mT(V)\ot
\mT(V)\ar[d]^{\tau\ot\tau} \\
L(A,\prov)\ot A \ar[rr]^{\can^{-1}} && A \ot A. }
\end{equation*}
Hence $\tau\ot\tau$ restricts to a surjection $L(\mT(V),\mT(V))\twoheadrightarrow
L(A,\prov)$ and the proposition follows.
\epf

\subsection{The graded Hopf algebra associated to a cocycle
deformation}\label{subsect:cocycles-graded}
Let $H$, $V$ be as in \eqref{eqn:H}, \eqref{eqn:V} and let $\B$ be a
pre-Nichols algebra over $V$. We set $\prov = \B \# H$ and let
$\sigma:\prov\ot \prov\to\ku$ be a 2-cocycle.
Let $\Fg = (F_n)_{n\geq 0}$ be the filtration of $\prov_\sigma$ induced by
the graduation of $\prov$. Then $\gr_{\Fg}
\prov_{\sigma}=\prov_\sigma=\prov$ as coalgebras.
Notice that, if $\B$
is a Nichols algebra, then $\Fg$ coincides with the coradical
filtration.

The items (a) and (b) of the next proposition are \cite[Theorem 2.7, Corollary 3.4]{MO}, see also \cite[Theorem 3.8]{AFGV}.

\begin{prop}\label{pro:deformations-liftings}
\begin{enumerate}
\renewcommand{\theenumi}{\alph{enumi}}
\renewcommand{\labelenumi}{(\theenumi)}
\item
There is an isomorphism of graded Hopf algebras $\gr_{\Fg} \prov_{\sigma} \simeq
\B'\#
H_{\sigma}$, for $\B'$ a pre-Nichols algebra over $V'\in
\YD{H_{\sigma}}{H_{\sigma}}$. Here $V'$ is the $H$-comodule $V$ with action
\begin{align}\label{eq:action-V-sigma}
x \rightharpoonup_{\sigma} v\,=\, \sigma(x\_{1}, v\_{-1})\, (x\_{2}\rightharpoonup v\_{0})\_{0}\, \sigma^{-1}((x\_{2}\rightharpoonup v\_{0})\_{-1},x\_{3})
\end{align}
for  $x\in H_{\sigma}$, $v\in V$; here $V$ is identified with a subspace of $T(V) \# H$,
with  $H$-action given by the adjoint.
Furthermore, the product in $\B'$ is given by $x\cdot
y=\sigma(x_{(-1)},y_{(-1)})x_{(0)}y_{(0)}$, for $x, y\in\B'$ homogeneous.
\item With the notation in \emph{(a)}, if $\B=\B(V)$ is the Nichols algebra of $V$, then
$\B'=\B(V')$.
\item Let $\prova\in\Cleft(\prov)$ with section $\gamma:\prov\to\prova$ and consider the
induced cocycle $\sigma(x\ot y)=\gamma(x\_1)\gamma(y\_1)\gamma^{-1}(x\_2y\_2)$, $x,y\in
\prov$, see \eqref{eq:cociclo-cleft}. Assume $\gamma_{|H}\in\Alg(H,\prova)$. Then
$\gr_{\Fg} \prov_{\sigma} \simeq \B \# H$.
\item Suppose that $H = \ku G$, $G$  a finite group. In particular, $H_\sigma=H$.
Let $\{x_1,\dots, x_\theta\}$ be a basis of $V$ with $x_i\in V^{g_i}$, $g_i\in G$, $1\leq
i\leq \theta$. If
\begin{align}\label{eqn:condicion cociclo-grupo}
\qquad \sigma(g, g_i) =\sigma(gg_ig^{-1},g), \quad g\in G,\ 1\leq
i\leq \theta,
\end{align}
then $V'=V\in\ydh$.
\end{enumerate}
\end{prop}
\pf
(a) By Remarks \ref{rems:filtrations-in-deformations} (d), $\gr_{\Fg}\prov$ is generated
by $H_{\sigma}\oplus (F_1/ H_{\sigma})$. Then
$\gr_{\Fg}
\prov_{\sigma} \simeq \B'\# H_{\sigma}$, where $\B'$ is a pre-Nichols algebra over $V' :=
\B'^{1}$. Since the
comultiplication is unchanged, $V'=V$ as $H_\sigma$-comodules and in $\gr_{\Fg}
\prov_{\sigma}$
\begin{align*}
x \rightharpoonup_{\sigma} v &= x\_{1} \cdot_{\sigma} v \cdot_{\sigma}
\Ss_\sigma(x\_{2})
\end{align*}
for all $x\in H_{\sigma}$, $v\in V'$. Using that $\Delta (v) = v\ot 1 + v\_{-1} \ot
v\_{0}$, \eqref{eq:action-V-sigma} follows as in the proof of \cite[Theorem 2.7]{MO}. Finally, if $x\in\B^n$ and $y\in\B^m$, then $x\cdot_\sigma
y=\sigma(x_{(-1)},y_{(-1)})x_{(0)}y_{(0)}$ plus terms of degree lesser than $m+n$. (b)
follows since the coalgebra structure is unchanged.
(c) follows since $\sigma_{|H\ot H}=\varepsilon\ot\varepsilon$. For (d), the
cocommutativity implies $H_{\sigma} = H$
and plugging
\eqref{eqn:condicion cociclo-grupo} into \eqref{eq:action-V-sigma}, we have $V' = V$.
\epf

\section{The strategy for computing  cocycle deformations}

Let $H$, $V$ be as in \eqref{eqn:H}, \eqref{eqn:V}. We explain how to compute cocycle deformations of $\toba(V)\# H$ that are liftings of $\toba(V)$ over $H$; 
depending on the context, this may eventually lead to all liftings. We
fix a minimal set $\Gc$ of homogeneous generators of $\J(V)$; $\Gc$ is
finite by assumption.


\smallbreak
Assume $\mL$ is a cocycle deformation, say $\mL=(\toba(V)\# H)_\sigma$. 
We seek for conditions for $\mL$ to be a lifting of $\toba(V)$ over $H$.  Let $\cA$ be a
$\toba(V)\#
H$-Galois object such that $\mL\simeq L(\cA,\toba(V)\# H)$.
By Proposition \ref{pro:deformations-liftings} (b), one has $\gr \mL\simeq \B(V')\#
H_\sigma$,
$V'\in\YD{H_{\sigma}}{H_{\sigma}}$. If $\sigma _{\vert H\ot H} \overset{\ast}=
\varep\ot\varep$, then $\gr \mL\simeq \B(V)\# H$ by
\eqref{eqn:cociclo-prod} and \eqref{eq:action-V-sigma}.
Now the equality $\ast$ is
achieved when the object $\cA$ is cleft with a section $\gamma:\toba(V)\# H\to \cA$
that satisfies
$\gamma_{|H}\in\Alg(H,\cA)$, by Proposition \ref{pro:deformations-liftings} (c).
In conclusion, we look for cleft objects $\cA$  with a section $\gamma:\toba(V)\# H\to \cA$ satisfying this property.

\subsection{Adapted stratifications}

 A stratification of $\Gc$ is a decomposition  as a
disjoint union $\Gc =  \Gc_0 \cup \Gc_1 \cup\dots \cup \Gc_{N}$.
For $0\le k \le N$, we set
\begin{align*}
\B_0 &:= T(V), &  \prov_0 &= T(V)\#H =\mT(V),
\\
\B_k &:= T(V) / \langle \Gc_0 \cup \Gc_1 \cup\dots \cup \Gc_{k-1}\rangle, &
\prov_k &=
\B_k\#H.
\end{align*}
Clearly $\prov_{N+1}=\toba(V)\#H$. Let $\pi_k: \mT(V) \to \prov_k$ be the canonical projection.
If $k < N$, then $\Gc_{k}$ identifies  with its image in $\B_k$, by
minimality of
$\Gc$. 

We say that the stratification $\Gc =  \Gc_0 \cup \Gc_1 \cup\dots \cup \Gc_{N}$
is \emph{adapted} when it satisfies the following properties:

\begin{enumerate}
	\item \emph{$\Gc_{k}$ is a basis of
a Yetter-Drinfeld submodule of
$\Pc (\B_k)$};
then $\langle\Gc_k\rangle$ is a Hopf ideal of $\B_k$, $\langle\Gc_k\rangle \#H$
is a Hopf
ideal of $\prov_k$ and therefore
$\prov_{k+1} \simeq \prov_k/\langle\Gc_k\rangle \#H$ is a Hopf algebra.
	
	\item \emph{$\Gc_{N}$ is a basis of a Yetter-Drinfeld
submodule
of $\B_{N}$ and $\ku\langle \Gc_{N}\rangle$ is a left coideal subalgebra of $\B_N$, but
not necessarily $\Gc_{N}\subset \Pc (\B_N)$}.
\end{enumerate}


\begin{Expls}\label{expls:adapted-strat}
\begin{enumerate}\renewcommand{\theenumi}{\alph{enumi}}
\renewcommand{\labelenumi}{(\theenumi)}
\item A standard choice is to take $\Gc_{k}$ such that (the image of) $\Gc_{k}$
is a basis of the subspace of $\Pc (\B_k)$ generated by all its homogeneous elements of degree $\ge 2$, for all $k$.
\item  Assume that $H = \ku\Gamma$, $\Gamma$ a
finite abelian group, or
more
generally that $V$ is a direct sum of one-dimensional Yetter-Drinfeld modules.
We may choose an adapted stratification with $\card \Gc_j = 1$ for all $j$.
\end{enumerate}
\end{Expls}

It is not always possible to choose a
stratification $\Gc =  \Gc_0 \cup \Gc_1 \cup\dots \cup \Gc_{N}$ in which $\card
\Gc_j = 1$ for each $j$, see the next example.

\begin{Ex}\label{exa:copointed} Keep the notation in Example \ref{exa:f5}. The
adapted stratification of $\B(X,q)$ considered in \cite[Theorem 5.4]{GIV} is:
\begin{align*}
\Gc_0 &= \{x_i^2:\,i\in\bF_5\},\\
\Gc_1 &= \{x_ix_j+x_{2j-i}x_i+x_{3i-2j}x_{2j-i}+x_jx_{3i-2j}:\,i,j\in\bF_5\},\\
\Gc_2 &= \{x_1x_0x_1x_0 + x_0x_1x_0x_1\}.
\end{align*}
The Yetter-Drinfeld submodule of $\toba_k$ generated by $\Gc_k$ is
simple, $k=0,1,2$.
\end{Ex}

\subsection{The strategy}\label{subsec:strategy} Fix $\Gc =  \Gc_0 \cup \Gc_1 \cup\dots
\cup \Gc_{N}$ an adapted stratification. The strategy consists of the following steps.

\medbreak
(1) \emph{Cleft extensions of bosonizations of pre-Nichols algebras}.
We shall construct recursively a set $\Lambda_{k}\subset \Cleft(\prov_k)$, for  all $k=0,\dots, N+1$.
 We start with $\prov_0=\mT(V)= T(V)\# H$ and the cleft object
$\prova_0=\mT(V)$ with section $\gamma_0=\id$. Set
$\Lambda_0=\big\{\mT(V)\big\}\subset\Cleft(\prov_0)$.

\medbreak
The recursive step is done in one of the following ways, for $k=0, \dots, N$:

\medbreak
(1a) We compute the algebra of left coinvariants $X_k := \,^{\co\prov_{k+1}}\prov_k$. Then
we compute
$\Alg_{\prov_{k}}^{\prov_{k}}(X_{k},\prova_{k})$, for each $\prova_{k}\in\Lambda_{k}$. \emph{For each
$\psi\in\Alg_{\prov_{k}}^{\prov_{k}}(X_{k},\prova_{k})$, we collect
$\cA_{k}/\cA_{k} \psi(X_{k}^+)$ in $\Lambda_{k + 1}$.}

\medbreak
Theorem \ref{thm:takeuchi-correspondance} may be useful to deal with the computation of
right coideal subalgebras in this step. However, it usually happens that the computation
of $X_k$, and \emph{a fortiori} that of
$\Alg_{\prov_{k}}^{\prov_{k}}(X_{k},\prova_{k})$, is too hard.
In such case, we take an alternative route.

\medbreak
(1b) We consider the subalgebra $Y_k := \ku\langle \Ss(\Gc_k)\rangle = \Ss(\ku\langle
\Gc_k\rangle)$ of $\prov_k$.
Since $\Gc =  \Gc_0 \cup \Gc_1 \cup\dots
\cup \Gc_{N}$ is an adapted stratification,
$\ku\langle \Gc_k\rangle$ is
a left coideal subalgebra of $\prov_k$;
hence $Y_k$ is a right coideal subalgebra of $\prov_k$. Also,
$$\cH_{k + 1}=\cH_{k}/\langle Y_{k}^+\rangle.$$ We then compute
$\Alg^{\prov_{k}}(Y_{k},\prova_{k})$, for each $\prova_{k}\in\Lambda_{k}$.  \emph{We collect
$\cA_{k}/\langle  \varphi(Y_{k}^+)\rangle$ in $\Lambda_{k+1}$, for each
$\varphi\in\Alg^{\prov_{k}}(Y_{k},\prova_{k})$ with}
\begin{align}\label{eq:gunther2}
\langle \varphi (Y_{k}^+)\rangle \neq \cA_{k}.
\end{align}

\medbreak The alternative (1a) has the advantage to the alternative (1b) to avoid the
checking  of \eqref{eq:gunther2}.
Note that
\begin{itemize}
  \item $\Lambda_{k}\subset \Cleft(\prov_k)$ by Theorem \ref{thm:gunther-teo4} in
(1a) or Theorem \ref{thm:gunther-extendido} in (1b).
\end{itemize}
Indeed, this holds for $k=0$ and
a recursive argument applies since the coradical of the successive quotients remains
unchanged. We can apply Theorem \ref{thm:gunther-teo4} because $\prov_k$ is
$\prov_{k+1}$-coflat, by Corollary \ref{cor:coflat}. This also implies that $\prov_k$ is
faithfully flat over $X_k$, by Theorem \ref{thm:takeuchi-correspondance}. As
$X_k=N(Y_k)$, see Remark \ref{rem:xk=nyk}, we can also apply Theorem
\ref{thm:gunther-extendido}.

\medbreak
(2) \emph{Deformations of pre-Nichols algebras}. We next compute the Hopf algebras
$L(\prova_{k},\prov_{k})$, for $\prova_{k}\in \Lambda_{k}$, $0<k\le N+1$. These are new
examples of Hopf algebras; they are quotients of $\mT(V)$ by Propositions \ref{pro:projection} and 
\ref{prop:properties-Ak}, which can be
computed using Proposition \ref{prop:L-cociente-gral} and Corollary \ref{cor:L-cociente}.

\medbreak
(3) \emph{Exhaustion}. The Hopf algebras
$L(\prova_{N+1},\prov_{N+1})$ for $\prova_{N+1}\in\Lambda_{N+1}$  are liftings of $\toba(V)$ over $H$, by Proposition
\ref{pro:deformations-liftings} (b) and (c).
We need to check whether this family of Hopf algebras
is an exhaustive list of liftings of $\toba(V)$ over $H$; for this  Theorem \ref{cor:generadors of ker phi} might apply under suitable conditions.

\begin{Rem}\label{rem:gunther-method}
A similar strategy is already proposed by G\"unther in \cite[page 4399]{G} to
compute the cleft objects of a pointed Hopf algebra $\overline{H}$ which is a
quotient of a pointed Hopf algebra $H$ for which $\Cleft(H)$ is known. He suggests to {\it
choose an easy decomposition} $H=H_1\twoheadrightarrow
H_2\twoheadrightarrow\dots\twoheadrightarrow H_n=\overline{H}$ in such a way that
$\Cleft(H_{i+1})$ is {\it easily computable} from $\Cleft(H_{i})$ using
\cite[Theorems 4 \& 8]{G}. He does not,
however, investigates how to find that decomposition or when the
method applies, nor relates this process with the lifting procedure or the classification
problem.
\end{Rem}

\subsection{Comments on $X_k$ and $Y_k$}\label{subsec:comments-Xk}

Let $\widetilde{X}_k = \prov_k^{\co \prov_{k+1}}$, $\widetilde{Y}_k =
\ku\langle \Gc_k\rangle$. The next picture describes the relation between these
subalgebras and the subalgebras of $\prov_k$ which are involved in
the steps (1a) and (1b).
\begin{align*}
\xymatrix{
X _k = \,^{\co \prov_{k+1}}\prov_k \ar@/^1.5pc/@{<->}[rr]^{\Ss} && \widetilde{X}_k=
\prov_k^{\co \prov_{k+1}} \ar@/^1.5pc/@{<->}[ll]_{\Ss^{-1}} \ar@{^{(}->}[r]
&\prov_k^{\co H} =
\toba_k \\
Y _k  = \ku\langle \Ss(\Gc_k)\rangle \ar@/^1.5pc/@{<->}[rr]_{\Ss} \ar@{^{(}->}[u] &&
\widetilde{Y}_k  = \ku\langle \Gc_k\rangle.\ar@{^{(}->}[u]
\ar@/^1.5pc/@{<->}[ll]^{\Ss^{-1}}&
}
\end{align*}
Indeed, $\Ss^2(\Gc_k) = \Gc_k$, for $k<N$. First, if $x\in \Pc (\B_k)$, then
$\Delta_{\prov_k}(x) = x\ot 1 + x\_{-1} \ot x\_0$, hence $\Ss(x) = -\Ss(x\_{-1})x\_0$
and $\Ss^2(x) = \ad\Ss(x\_{-1})(x\_0)$. So, $\Gc_k$, being a Yetter-Drinfeld submodule of
$\Pc (\B_k)$, is stable under $\Ss^2$.

\begin{Rem}\label{rem:xk=nyk}
$X_k=N(Y_k)$, cf. page \pageref{def:NX}. 

Indeed, let $B$ be the subalgebra generated by $h_{(1)}y \Ss(h_{(2)})$, $h\in \prov_k$,
$y\in \widetilde{Y}_k$. By Corollary \ref{cor:flat},  $\prov_k$ is
right $N(Y_k)$-faithfully flat. Now we invoke Theorem \ref{thm:takeuchi-correspondance}: ${\mathcal
I}(N(Y_k))=\langle \Gc_k\rangle$; but ${\mathcal X}(\langle
\Gc_k\rangle)=X_k$, as $\prov_k$ is $\prov_{k+1} \simeq \prov_k/\langle\Gc_k\rangle
\#H$-coflat by Corollary \ref{cor:coflat}.
\end{Rem}

As said, the computation of the algebra of coinvariants  $X_k$ might be hard; a
potentially easier instance is when $X_k = Y_k$.
We analyze when this could happen in the following Remark.

\begin{Rem}\label{rem:Y_k=X_k}
The following are equivalent:
\begin{enumerate}
	\item  $X_k = Y_k$;
	
	\item $Y_k$ is normal;
	
	\item  For all $y\in \Gc_k$, $x\in V$,
\begin{equation}\label{eq:adV-Y_k}
\ad_r(x)(\Ss(y)) = \Ss(x\_{-1}) \Ss(y) x\_0 - \Ss(x\_{-1})x\_0 \Ss(y) \in Y_k;
\end{equation}

\item   $x\Ss(y)  - \Ss(y)x \in Y_k$, for all $y\in \Gc_k$, $x\in V$.
\end{enumerate}
\end{Rem}

\pf (1) $\Rightarrow$ (2): $X_k$ is normal. (2) $\Rightarrow$ (1): $Y_k=N(Y_k)=X_k$ by
Remark \ref{rem:xk=nyk}. Clearly, (2) $\Rightarrow$ (3). (3) $\Rightarrow$ (2): We have to
prove that
$\ad_r(x)(z)\in Y_k$ for all $x\in \prov_k$, $z\in Y_k$.
Since $\ad_r(x)(zz') = \ad_r(x\_1)(z)\ad_r(x\_2)(z')$, it is enough to consider $z\in
\Ss(\Gc_k)$;
since $\ad_r(xx')(z) = \ad_r(x')\ad_r(x)(z)$, it is enough to consider $x\in H$ or $x\in
V$. If $x\in H$, $u = \Ss^{-1}(x)$ and $y\in \Gc_k$,
then $\ad_r(x)(\Ss(y)) = \Ss(\ad_\ell(u) (y))  = \Ss(u \cdot y) \in \Ss(\Gc_k)$.  It only
remains the case $z\in \Ss(\Gc_k)$ and $x\in V$,
which is \eqref{eq:adV-Y_k}.

(3) $\Rightarrow$ (4): $x\Ss(y) = x\_{-2} \Ss(x\_{-1})x\_0\Ss(y)
\overset{\eqref{eq:adV-Y_k}}= x\_{-2}\Ss(x\_{-1}) \Ss(y) x\_0 + Y_k = \Ss(y) x + Y_k$.
The converse implication is similar.
\epf

\subsection{Properties of $\prova_{k+1}$}\label{subsec:properties-Ak}

Fix $k\geq 0$ and $\prova_{k+1}\in\Lambda_{k+1}$ which is a quotient of
$\prova_{k}\in\Lambda_{k}$. We collect some information about the algebra
$\prova_{k+1}$. 
We start with some general considerations.

\begin{Rems}\label{rem:comod-alg-boson}
Let $H$ be a Hopf algebra and $\toba$ be a Hopf algebra in $\ydh$. Set $\prov =
\toba\# H$ with projection and inclusion maps $\xymatrix{\prov \ar@<0.4ex>@{->}[r]^{\pi}
& H \ar@<0.4ex>@{->}[l]^{\iota}}$.
Let $\prova \in \Cleft(\prov)$ with section $\gamma:\prov\to
\prova$; assume that $\gamma_{\vert H}\in\Alg(H,\prova)$.  Both $\prov$ and $\prova$ are $H$-comodules via $\pi$. Then
\begin{enumerate}
\renewcommand{\theenumi}{\alph{enumi}}
\item\label{item-b-rem:comod-alg-boson}  $\prova$ is a cleft extension of $H$. Moreover
$\prova\simeq \Ee \# H$, where
$\Ee = \prova^{\co H}$, and $p:\prova\rightarrow\Ee$,
$p(x)=x\_{0}\gamma^{-1}\iota\pi(x\_{-1})$ is an $H$-module projection. 
\item\label{item-a-rem:comod-alg-boson} Let $S\subseteq\Ee$ be an $H$-submodule. Then
$\langle S\rangle_{\prova} = \langle
S\rangle_{\Ee}\# H$
and consequently $\prova/\langle S\rangle_{\prova} \simeq (\Ee/\langle S\rangle_{\Ee})\#
H$.
\item\label{item-c-rem:comod-alg-boson} Let $I\subset\prova$ be an ideal and
$H$-subcomodule. Then $I=I^{\co H}\# H$.
\item\label{item-d-rem:comod-alg-boson} Let $S\subset\prova$ be an $H$-submodule and 
$H$-subcomodule. Then $\langle S\rangle_{\prova}=\langle p(S)\rangle_{\Ee}\# H$.
\end{enumerate}
\end{Rems}

\pf (a) Clearly, $\gamma$ is an
$H$-colinear
map and thus $\gamma\iota:H\to \prova$ is
a section. Hence $\prova\simeq \Ee \#_{\sigma}
H$ and $\sigma=\varep$ since $\gamma_{\vert H}\in\Alg(H,\prova)$, cf.
\eqref{eq:cociclo-cleft}. Last sentence is \cite[Lemma
7.2.6]{Mo}. (b) is easy. \eqref{item-c-rem:comod-alg-boson} $\langle I^{\co
H}\rangle_{\prova}=\langle I^{\co H}\rangle_{\Ee}\# H\subset I$ by (b). If $x\in I$, then
$p(x)\in I^{\co H}$ and thus $x=p(x\_{0})\gamma\iota(\pi(x\_{1}))\in
\langle I^{\co H}\rangle_{\prova}$. Finally $\langle I^{\co H}\rangle_{\Ee}=I^{\co H}$.
\eqref{item-d-rem:comod-alg-boson} By hypothesis, $p(S)$ is an $H$-submodule. Then
$S\subseteq\langle p(S)\rangle_{\prova}=\langle p(S)\rangle_{\Ee}\# H\subseteq\langle
S\rangle_\prova$, by (b), and the equality holds.
\epf

\begin{lema}\label{lem:Ainjective}
Let $H\subset\prov$ be Hopf algebras where $H$ is finite-dimensional and semisimple.
If $\prova \in \Cleft(\prov)$, then $\prova$ is an injective object in
$\mathcal{YD}^{\cH}_{H}$.
\end{lema}

\pf
Recall that $\mathcal{YD}^{\cH}_{H}\simeq \mathcal{M}^{L}$, with
$L=\cH\blacktriangleright\hspace*{-.2cm}\blacktriangleleft_\tau H^{*\,\cop}$, see
\cite[Exercise 7.2.16]{M}. Here $\tau=\sum_i\Ss(e^i)\ot e_i\in H^{*\,\cop}\ot \cH$, for
dual bases $\{e_i\}$, $\{e^i\}$ of $H$ and $H^*$. In particular, $L\simeq \cH\ot
H^{*\,\cop}$ as
algebras and there is a Hopf algebra projection $L\twoheadrightarrow
\cH$, see {\it loc.~cit.} for details.

As $\prova$ is cleft, it is $\cH$-coflat. As $L\twoheadrightarrow \cH$ is a cosemisimple
coextension, then $\prova$ is also $L$-coflat and the lemma follows.
\epf

The following is a snapshot of the Strategy:
$$
\xymatrix{
\mT(V)\ar@{->}[rr]^{\gamma_0=\id}\ar@{->}[d]^{\pi_{k}}\ar@/^-2pc/@{->}[dd]_{\pi_{k+1}}
^{\, \circlearrowright} &&
\mT(V)\ar@{->}[d]_{\tau_{k}}\ar@/^2pc/@{->}[dd]^{\tau_{k+1}}_{
\circlearrowleft\, } \ar@{~>}[rr]
&&\mT(V)&\\
\prov_{k}\ar@{->}[rr]^{\gamma_{k}}\ar@{->}[d]^{\pi'_{k+1}} &&
\prova_{k}\ar@{->}[d]_{\tau'_{k+1}}&\ar@{~>}[r]&L(\prova_{k},\prov_{k})&\\
\prov_{k+1}\ar@{->}[rr]^{\gamma_{k+1}} && \prova_{k+1}\ar@{~>}[rr]&&L(\prova_{k+1},\prov_{k+1})&
}
$$
Here $\pi_{k}$, $\pi_{k+1}$ and $\pi'_{k}$ are the natural projections of Hopf algebras
and $\tau_{k}'$ is the natural projection of algebras which is also $\prov_{k+1}$-colinear
via $\pi'_{k+1}$. The epimorphisms  $\tau_k: \mT(V) \to \prova_k$ are defined recursively
as follows: If $k=1$, we take $\tau_1=\tau_1'$. Given $\tau_{k}$, we set
$\tau_{k+1}=\tau'_{k+1}\tau_{k}$. Notice that each $\tau_k$ is a morphism of right
$\prov_k$-comodule algebras.

The sections $\gamma_{k}$, $\gamma_{k+1}$ are
introduced in the next proposition.

\begin{prop}\label{prop:properties-Ak}
\begin{enumerate}\renewcommand{\theenumi}{\alph{enumi}}
\renewcommand{\labelenumi}{(\theenumi)}

\item\label{item-b-prop:properties-Ak}
The Miyashita-Ulbrich action \eqref{eq:Miyashita-Ulbrich} on $\prova_k$ is:
\begin{align}\label{eq:Miyashita-Ulbrich-tau}
a \leftharpoonup \pi_k(x) & = \tau_k(\Ss(x\_1))a\tau_k(x\_2),&  &a\in \prova_k,\, x\in
\mT(V).
\end{align}

\item\label{item-c-prop:properties-Ak}
There is a section $\gamma_k:\prov_k \to\prova_k$ such that
${\gamma_k}_{|H}\in\Alg(H,\prova_k)$, ${\gamma_k}_{|H}={\tau_k}_{|H}$ and
$\gamma_k(xh)=\gamma_k(x)\gamma_k(h)$ for all $x\in\B_k$, $h\in H$. 
\smallbreak
\item\label{item-cc-prop:properties-Ak} If $H$ is semisimple, then the above
$\gamma_k:\prov_k\to \prova_k$ is a morphism of right $H$-modules and $\gamma_k(hx)=\gamma_k(h)\gamma_k(x)$ for all $x\in\B_k$, $h\in H$.
\smallbreak
\item\label{item-d-prop:properties-Ak}
$\prova_k$ is a cleft extension of $H$, $\prova_k\simeq \Ee_k \# H$, where
$\Ee_k = \prova_k^{\co H}$ and
\begin{align*}
\tau_k(T(V)\# 1)=\gamma_k(\B_k\# 1)=\tau'_k(\Ee_{k-1}\# 1)= \Ee_k\# 1.
\end{align*}
In particular, $\Ee_{k}\simeq\Ee_{k-1}/(\ker\tau'_k)^{\co H}$.
\end{enumerate}
\end{prop}
\pf

\eqref{item-b-prop:properties-Ak} Let $x\in \mT(V)$. We compute
\begin{align*}
\can(\tau_k(\Ss(x\_1))\ot\tau_k(x\_2))&=\tau_k(\Ss(x\_1))\tau_k(x\_2)\_0\ot
\tau_k(x\_2)\_1\\
\notag &=\tau_k(\Ss(x\_1))\tau_k(x\_2)\ot \pi_k(x\_3)=1\ot \pi_k(x).
\end{align*}
Then $\can^{-1}(1\ot \pi_k(x))=\tau_k(\Ss(x\_1))\ot \tau_k(x\_2)$ and
\eqref{eq:Miyashita-Ulbrich-tau} follows by \eqref{eq:Miyashita-Ulbrich}.

\eqref{item-c-prop:properties-Ak} First, we may choose $\gamma_0=\id$, then the statement
holds trivially for $k=0$. We now proceed by induction, assume it holds for $k$. The
inductive step follows as \cite[Theorem 4.2]{Sc} in this
setting: Notice that $\prova_{k}$ is an injective $\prov_{k+1}$-comodule, since
$\prov_{k}$ is $\prov_{k+1}$-coflat and $\prova_{k}$ is $\prov_{k}$-coflat.
Thus, as ${\gamma_{k}}_{|H}:H\to \prova_{k}$ is $\prov_{k+1}$-colinear, there exists a
$\prov_{k+1}$-colinear map $\omega:\prov_{k+1}\to \prova_{k}$ such that
$\omega_{|H}={\gamma_{k}}_{|H}$. By \cite[Lemma 14]{T1} $\omega$
is convolution-invertible, since its restriction to $H$ is. Also
$\tau'_{k+1}\,\omega_{|H}\in\Alg(H,\prova_{k+1})$. Then the section
$\gamma_{k+1}:\prov_{k+1}\to  \prova_{k+1}$ is defined by
$$
xh\longmapsto\tau'_{k+1}\,\omega(x)\,\tau'_{k+1}\,\omega(h),\qquad x\in\B_{k+1},\,h\in H.
$$
Note that
$\gamma_{k+1}(h)=\tau'_{k+1}\,\omega(h)=\tau'_{k+1}\gamma_{k}(h)=\tau'_{k+1}\tau_{k}(h)=\tau_{k+1}(h)$, $h\in H$.

\eqref{item-cc-prop:properties-Ak} If $H$ is semisimple, then $\prova_{k}$ is injective
in $\mathcal{YD}^{\cH_{k}}_{H}$ by Lemma \ref{lem:Ainjective}. 
The  proof of \eqref{item-c-prop:properties-Ak} \emph{mutatis mutandis} shows the first claim. If $x\in\B_{k+1}$, $h\in H$, then
$\gamma_{k+1}(hx)=\gamma_{k+1}(\ad_r(\Ss^{-1}(h_{(2)}))(v)h_{(1)})=\gamma_{k+1}(h)\gamma_{
k+1 } (x)$.
\smallbreak

\eqref{item-d-prop:properties-Ak} By Remark \ref{rem:comod-alg-boson}
\eqref{item-b-rem:comod-alg-boson}, $\prova_k$ is $H$-cleft and
$\prova_k\simeq \Ee_k \# H$. As $\B_k\#1=\cH_k^{\co H}$, $T(V)\#1=\mT(V)^{\co H}$,
$\Ee_{k-1}\#1=\prova_{k-1}^{\co H}$,
we have $\tau_k(T(V)\# 1)$, $\gamma_k(\B_k\# 1)$, $\tau'_k(\Ee_{k-1}\# 1)\subseteq \Ee_k\#
1$ since all $\gamma_k$, $\tau_k$ and $\tau'_k$ are $H$-colinear. The equality and the last
assertion of \eqref{item-d-prop:properties-Ak} follow from Remark
\ref{rem:comod-alg-boson} \eqref{item-c-rem:comod-alg-boson}.
\epf

We fix the following setting\label{page:setting}:
\begin{itemize}
\item We denote by $u_i\in \mP( \B_k)$, $1\le i \le n$, the elements
of $\Gc_k$.  We set $v_i=\Ss(u_i)\in Y_k$, $1\leq i\leq n$. Let $U$, resp. $W$, be the
linear
span of $\{u_i\}_{1 \leq i\leq
n}$, resp. $\{v_i\}_{1 \leq i\leq
n}$. Notice that $U\in\ydh$, $W\in \mathcal{YD}_H^H$.
\item Let $\{e_{ij}\}_{1
\leq i,j\leq n}\subset H$ be the set of comatrix elements associated to
$U$ and $\{u_i\}_{1\le i \le n}$, see \eqref{eqn:comatrix}. Then
$$
\Delta(v_i)=\sum_{j=1}^nv_j\ot\Ss(e_{ij})+1\ot v_i,\quad1\leq i\leq n.
$$
\item We set $C$ the
subcoalgebra of $H$ generated by $\{\Ss(e_{ij})\}_{1 \leq i,j\leq n}$.
\item  We fix a
section $\gamma_k:\cH_{k}\to \cA_k$ such that $\gamma_k(xh)=\gamma_k(x)h$, $x\in\B_k$,
$h\in H$, see Proposition \ref{prop:properties-Ak}. If $H$ is semisimple we assume
moreover that $\gamma_k$ is $H$-linear.
\end{itemize}

By Proposition \ref{prop:properties-Ak} \eqref{item-c-prop:properties-Ak} we can identify
$H$
with $\tau_k(H)=\gamma_k(H)\subset \cA_k$.

\begin{lema}\label{lem:caracterizar f-gral} 
\begin{enumerate}\renewcommand{\theenumi}{\alph{enumi}}
 \renewcommand{\labelenumi}{(\theenumi)}
\item\label{item:a:lem:caracterizar f-gral} Let
$\varphi\in\Alg^{\prov_k}(Y_k,\prova_k)$. There are
$\{c_i\}_{1\leq i\leq n}\subset\ku$  with
\begin{align}\label{eqn:f-gral}
\varphi(v_i)=\gamma_k(v_i)+\sum_{j=1}^nc_j\, \Ss(e_{ij}).
\end{align}
\item\label{item:b:lem:caracterizar f-gral} If $H$ is semisimple and $\varphi$ is
$H$-linear, then
$(\varphi-\gamma_k)_{|W}:W\rightarrow\prova_k$ is a
morphism in $\mathcal{YD}_H^H$ whose image is contained in $C$.
\end{enumerate}
\end{lema}
\pf  
\eqref{item:a:lem:caracterizar f-gral} As $H$ is cosemisimple, $U=\bigoplus_lM_l$
where each $M_l$ is a simple $H$-comodule. We can
assume that each $u_i$ belongs to some
$M_l$, up to changing the basis of $U$. For each $i$, we restrict to the subcomodule
$M_l$ with $u_i\in M_l$ and consider the corresponding comatrix elements
$\{e_{ij}\}_{i,j}$; they are linearly independent. To simplify the notation, assume
$U=M_l$ is simple.

Set
$b_i=\varphi(v_i)-\gamma_k(v_i)$ for $1\leq i\leq n$. Since $\varphi$
and $\gamma_k$ are $\prov_k$-colinear,
$\rho(b_i)=\sum_{j=1}^nb_j\ot\Ss(e_{ij})$. Then $\{b_1, \dots,
b_n\}\subset H$ since $H$ is the socle of $\prova_k$. Moreover, $\{b_1, \dots,
b_n\}\subset C$. Now, we write $b_\ell=\sum_{i,j=1}^{n}c_{ij}^\ell \Ss(e_{ij})$ for all $1
\leq \ell\leq n$, where $c_{ij}^\ell \in\ku$. We have
\begin{align*}
\Delta(b_\ell)&=\sum_{i,j}c_{ij}^\ell\Delta(\Ss(e_{ij}))=\sum_{i,j,s}c_{ij}^\ell\,
\Ss(e_{sj})\ot \Ss(e_{is}),\\
\rho(b_\ell)&=\sum_{s}b_s\ot \Ss(e_{\ell s})=\sum_{i,j,s}c_{ij}^s\, \Ss(e_{ij})\ot
\Ss(e_{\ell s}).
\end{align*}
Recall that $\Delta(b_\ell)=\rho(b_\ell)$ since $b_\ell\in H$, then $c_{ij}^\ell=0$, if
$\ell\neq i$, and hence $c_{\ell j}^\ell=c_{s j}^{s}$ for all $1 \leq \ell,s\leq n$.
Therefore we set $c_j=c_{\ell j}^\ell$ for each $1 \leq j\leq n$.
\eqref{item:b:lem:caracterizar f-gral} follows from Proposition \ref{prop:properties-Ak}
\eqref{item-cc-prop:properties-Ak}. 
\epf

\subsection{The shape of $L(\prova_{k+1},\prov_{k+1})$} We keep the setting above and also: 
\begin{itemize}
\item We fix $\varphi\in\Alg^{\cH_{k}}(Y_k,\cA_k)$ with $c_j\in\ku$, $1\leq j \leq n$, as
in \eqref{eqn:f-gral}. 

\smallbreak
\item We assume that $\cA_{k+1}=\cA_k/\langle \varphi(v_i)\rangle_{1\leq i \leq n}\neq 0$,
then $\prova_{k+1}\in\Lambda_{k+1}$. 

\smallbreak
\item We fix a Hopf algebra $\mL_{k}$ such that $\prova_k$ is a $(\mL_{k},
\cH_{k})$-biGalois
object. Let
$\vartheta:
L(\prova_k,\prov_k)\to \mL_{k}$ be the isomorphism in \eqref{eqn:f-schauenburg}.

\smallbreak
\item We identify $H\hookrightarrow \mL_{k}$ as a Hopf subalgebra via $\vartheta(\id\ot\Ss)\Delta$
since ${\tau_k}_{|H}={\gamma_k}_{|H}=\id_H$.
\end{itemize}

\smallbreak

We now describe $L(\cA_{k+1},\cH_{k+1})$ as a quotient of $\mL_{k}$. 

\begin{prop}\label{prop:L-cociente-gral}
$L(\cA_{k+1},\cH_{k+1})\simeq\mL_{k}/\langle \widetilde v_i -
c_i+\sum_{j=1}^nc_j \Ss(e_{ij})\rangle_{1\leq i \leq n}$ where $\widetilde v_i\in
\mL_{k}$, $1\leq i \leq n$, is such that
\begin{align}\label{eqn:primitivo2-gral}
 \widetilde v_i\ot
1_{\cA_k}=\sum_{t}\gamma_k(v_t)\_{-1}\ot\gamma_k(v_t)\_{0}\gamma_k^{-1}\Ss(e_{it}
)+1\ot\gamma_k^{-1}(v_i).
\end{align}
\end{prop}
\pf 
We may assume $\mL_{k}=L(\prova_k,\prov_k)\subset \prova_k\ot \prova_k$. The general case
follows by applying $\vartheta$. 
Set 
\begin{align*}
& \overline{v}_i=(\gamma_k\ot\gamma_k^{-1})\Delta(v_i), &&
E_{ij}=(\gamma_k\ot\gamma_k^{-1})\Delta(\Ss(e_{ji}))
\end{align*}
for all $1\leq i,j\leq n$. Let $J=\big\langle \overline v_i -
c_i+\sum_{j=1}^nc_jE_{ji}\big\rangle_{1\leq i \leq n},
$
notice that this is a Hopf ideal since $(\gamma_k\ot\gamma_k^{-1})\Delta$ is an
(injective) coalgebra map by \eqref{eqn:cleft-can-1} and \eqref{eqn:comultiplication
L}. Set $\mL_{k+1}=L(\cA_k,\cH_{k})/J$. We have
to show that $L(\cA_{k+1},\cH_{k+1})\simeq \mL_{k+1}$. By \cite[Theorem 3.5]{S}, it
suffices
to prove the following statement.
\begin{claim}
$\cA_{k+1}$ is a $(\mL_{k+1},\cH_{k+1})$-biGalois object. 
\end{claim}
Set $I=\langle\varphi(v_i)\rangle_{1\leq i \leq n}\subset \cA_k$. Let
$\lambda:\cA_k\to L(\cA_k,\cH_{k})\ot \cA_k$ be the
coaction as in \eqref{eqn:comultiplication L}. Then $\lambda(I)\subset
L(\cA_k,\cH_{k})\ot I + J\ot \cA_k$ and thus $\lambda$ induces a coaction
$\lambda':\cA_{k+1}\to \mL_{k+1}\ot \cA_{k+1}$ such that $\cA_{k+1}$ is a left
$\mL_{k+1}$-comodule algebra. Indeed, it is straightforward to see that

\begin{align*}
\lambda&(\varphi(v_i))=\lambda_k\Big(\gamma_k(v_i)+\sum_{j=1}^nc_j\, \Ss(e_{ij})\Big)\\
&=\sum_{t}\Big(\gamma_k(v_t)+\sum_{j=1}^nc_j\,
\Ss(e_{tj})\Big)\ot\can^{-1}(1\ot\Ss(e_{it}))+1\ot\can^{-1}(1\ot v_i)\\
&=\sum_{t,s}\gamma_k(v_t)\ot\gamma_k^{-1}\Ss(e_{st})\ot\gamma_k\Ss(e_{is})+\sum_{s}
1\ot\gamma_k^{-1}(v_s)\ot\gamma_k\Ss(e_{is})\\
&\qquad +
\sum_{s,t,j}c_j\,
\Ss(e_{tj})\ot\gamma_k^{-1}\Ss(e_{st})\ot\gamma_k\Ss(e_{is})+1\ot1\ot\gamma_k(v_i)\\
&=\sum_{s}\Big(\overline
v_s+\sum_{j}c_jE_{js}-c_s\Big)\ot\Ss(e_{is})+1\ot1\ot\Big(\gamma_k(v_i)+\sum_sc_s\Ss(e_{
is})\Big)\\
\end{align*}
This also shows that $\cA_{k+1}$ is a $(\mL_{k+1},\cH_{k+1})$-bicomodule algebra. Let
$\can: \cA_k\ot \cA_k\to L(\cA_k,\cH_{k})\ot \cA_k$ be the Galois map. To conclude, we need to show
that $$\can(I\ot \cA_k+\cA_k\ot I)=J\ot \cA_k+L(\cA_k,\cH_{k})\ot I.$$ 
Indeed, by the above
computation, $\can(I\ot \cA_k+\cA_k\ot
I)\subseteq J\ot \cA_k+L(\cA_k,\cH_{k})\ot I$. Now, we show that $\can^{-1}(J\ot
\cA_k+L(\cA_k,\cH_{k})\ot I)\subseteq I\ot
\cA_k+\cA_k\ot I$, using \eqref{eq;can-1}. To this end, it is enough to check that $
\can^{-1}((\overline v_i - c_i+\sum_{j=1}^nc_jE_{ji})\ot \cA_k)\subseteq I\ot
\cA_k+\cA_k\ot I
$
since the other inclusion is straightforward. This is a consequence of the following
computation:
\begin{align*}
\can^{-1}&\Big((\overline v_i - c_i+\sum_{j=1}^nc_jE_{ji})\ot
a\Big)=\sum_{t}\gamma_k(v_t)\ot\gamma_k^{-1}\Ss(e_{it})a+1\ot\gamma_k^{-1}(v_i)a\\
&\qquad \qquad \qquad -1\ot
c_ia+\sum_{j,t}c_j\, \gamma_k\Ss(e_{tj})\ot\gamma_k^{-1}\Ss(e_{it})a\\
\overset{(\star)}=&\sum_{t}\Big(\gamma_k(v_t)+\sum_{j}c_j\,\Ss(e_{tj})\Big)\ot\gamma_k^{
-1}\Ss(e_{it})a\\
&\qquad\qquad \qquad  -1\ot\Big(\sum_{t}\gamma_k(v_t)\gamma_k^{-1}\Ss(e_{it})+c_i\Big)a\\
\overset{(\star\star)}=&\sum_{t}\Big(\gamma_k(v_t)+\sum_{j}c_j\,\Ss(e_{tj}
)\Big)\ot\gamma_k^{-1}\Ss(e_{it})a\\
&\qquad\qquad \qquad  -1\ot
\sum_t\Big(\gamma_k(v_t)\sum_{j}c_j\,\Ss(e_{tj})\Big)\gamma_k^{-1}\Ss(e_{it})a
\end{align*}
for each $a\in \cA_k$. In the computation above, $(\star)$, resp. $(\star\star)$, follows
since

\begin{align*}
\sum_t\gamma_k(v_t)\gamma_k^{-1}\Ss(e_{it})+\gamma_k^{-1}(v_i)&=0,\quad \text{resp.}\\
\sum_t\Big(\gamma_k(v_t)+\sum_{j}c_j\,\Ss(e_{tj})\Big)\gamma_k^{-1}\Ss(e_{it}
)&=\sum_t\gamma_k(v_t)\gamma_k^{-1}\Ss(e_{it})+\sum_{j}c_j\,\varepsilon(e_{ij}).
\end{align*}
This ends the proof of the claim.
\epf

\subsection{Adapted stratifications with a stratum of skew-primitive elements}\label{subsec:coinvariants}

We add the following assumption to the setting of page
\pageref{page:setting}:

\begin{itemize}
\item We assume that $\Gc_k$ is composed of skew-primitive
elements in $\cH_k$. Explicitly, $u_i\in\mP_{g_i,1}(\cH_{k})$ for some $g_i\in
G(\cH_{k}) = G(H)$. In particular, $v_i=u_ig_i^{-1} \in\mP_{1,g_i^{-1}}(\cH_{k})$, so
$Y_k=\ku\langle v_i\rangle_{1\leq i \leq n}$.
\end{itemize}

\begin{Rem}\label{rem:caracterizar f} Let $\varphi\in\Alg^{\cH_{k}}(Y_k,\cA_k)$. Lemma
\ref{lem:caracterizar f-gral} \eqref{item:a:lem:caracterizar f-gral} in this context says
that there exist
$c_i\in\ku$, $1\leq i\leq n$, such that
\begin{align*}
\varphi(v_i)=\gamma(v_i)-c_i\, g^{-1}_i.
\end{align*}
\end{Rem}

Let
$\varphi\in\Alg^{\cH_{k}}(Y_k,\cA_k)$, $c_i\in\ku$ be as in Remark \ref{rem:caracterizar
f}.
Assume that $\cA_{k+1}=\cA_k/\langle \varphi(v_i)\rangle_{1\leq i \leq n}=\cA_k/\langle
\gamma(u_i)
-c_i\rangle_{1\leq i \leq n}\neq0$. Let $\mL_{k}$ be a Hopf algebra such that $\prova_k$
is a $(\mL_{k}, \cH_{k})$-biGalois object.

Proposition \ref{prop:L-cociente-gral} is formulated in this context as follows, compare
with \cite[Lemma 11]{G}.

\begin{Cor}\label{cor:L-cociente}
$L(\cA_{k+1},\cH_{k+1})\simeq \mL_k/\langle \widetilde u_i -
c_i(1-g_i)\rangle_{1\leq i \leq n}$, where $\widetilde u_i\in \mP_{g_i,1}(\mL_k)$ is such
that
\begin{align}\label{eqn:primitivo2}
 \widetilde u_i\ot 1_{\cA_k}=\gamma_k(u_i)_{(-1)}\ot \gamma_k(u_i)_{(0)}-g_i\ot
\gamma_k(u_i).
\end{align}
\end{Cor}
\pf
Follows by Proposition \ref{prop:L-cociente-gral}. The fact that $\widetilde u_i\in
\mP_{g_i,1}(\mL_k)$ follows since $\widetilde u_i=\vartheta
\big((\gamma_k\ot\gamma_k^{-1})\Delta(u_i)\big)$ and $(\gamma_k\ot\gamma_k^{-1})\Delta$
is a coalgebra map.
\epf

\subsubsection{A stratum generated by one-dimensional
submodules}\label{subsec:coinvariants2}

In this part we refine the previous setting as follows:

\begin{itemize}
\item We assume there is a family of
YD-pairs $(g_i, \chi_i) \in G(H) \times \Alg(H, \ku)$ cf.
\eqref{eq_yd-pair}, $i=1,\dots,n$, such that $u_i\in \mP(\B_k)^{\chi_i}_{g_i} - 0$. We
also assume that $u_i$ is homogeneous of degree $d_i \geq 2$. We identify $u_i$ with
$u_i\#1
\in \mP_{g_i,1}(\cH_{k})$. Recall that $v_i=u_ig_i^{-1}$.
\end{itemize}

For completeness, we include  the proof of the following well-known result.

\begin{lema}\label{lem:X}
Assume $\car \ku=0$. Let
$q_i= \chi_i(g_i)$, $N_i = \ord q_i$. Then $\ku\langle v_i \rangle$ is either a
polynomial algebra or a polynomial algebra truncated at $N_i$ (in case $N_i\ge 2$).
\end{lema}

Clearly, $q_i$ is a root of 1. If $q_i = 1$, then $\ku\langle v_i \rangle$ is always a
polynomial algebra.
For $q_i\neq 1$, it is possible to check  whether $\ku\langle v_i \rangle$ is truncated
in specific
examples.

\pf Notice that $\ku\langle v_i \rangle$ is a Hopf algebra in $\hyd$, with braiding
determined by $c(v_i\ot
v_i) =
q_i^{-1}(v_i\ot v_i)$.
Consider the polynomial algebra $\ku[T]$ as a braided Hopf algebra with the analogous
braiding.
Then the kernel of the epimorphism $\varsigma: \ku[T] \to \ku\langle v_i \rangle$, given
by $T\mapsto v_i$,
 is an homogeneous Hopf ideal of the braided Hopf algebra $\ku[T]$ spanned by a primitive
element.
Thus $\ker \varsigma = 0$ or $\langle T^M\rangle$ for some $M$, but $T^M$ is primitive
only when $M = 1$ or $M =\ord q_i^{-1} = N_i$.
\epf

In the following lemma we study the set $\Alg_{\cH_{k}}^{\cH_{k}}(X_k,\prova_k)$, which
is required for Step (1a) of the Strategy.

\begin{lema}\label{lem:caracterizar f-a}
Let $\psi_1,\psi_2\in\Alg_{\cH_{k}}^{\cH_{k}}(X_k,\prova_k)$. Fix $j$, $1\leq j\leq n$,
with
${\chi_j}_{|G(H)}\neq
\varep$.
 \begin{enumerate}\renewcommand{\theenumi}{\alph{enumi}}
 \renewcommand{\labelenumi}{(\theenumi)}
  \item\label{lem:caracterizar f-a-item-a}
 $\psi_1(v_j)=\psi_2(v_j)$.
 \item\label{lem:caracterizar f-a-item-b}
If $H$ is semisimple, then $\psi_1(v_j)=\psi_2(v_j)=\gamma(v_j)$.
 \end{enumerate}
\end{lema}
\pf (a) Since $(\gamma(v_j)\leftharpoonup t^{-1}-\chi_j(t)\gamma(v_j))g_j\in \cA_k^{\co
\cH_{k}}=\ku$, we see that there exists a map $a_j:G(H)\to \ku$  such that
\begin{align}\label{eq:aj}
 \gamma(v_j)\leftharpoonup t^{-1}=\chi_j(t)\gamma(v_j)+a_j(t)g_j^{-1}, \qquad t\in
G(H).
\end{align}
Note that  $a_j(1)=0$ and $a_j(ts)=a_j(s)+\chi_j(s)a_j(t)$.
Let $\varphi_i={\psi_i}_{|Y_k}\in\Alg^{\cH_{k}}(Y_k,\cA_k)$. Then
$\psi_i(v_j)=\gamma(v_j)-c_j^{(i)}g^{-1}_j$ by Remark \ref{rem:caracterizar
f}, for some $c_j^{(i)}\in\ku$, $i=1,2$. As $\psi_i$ is $\cH_{k}$-linear, we have
\begin{align*}
 0 =\psi_i(v_j)\leftharpoonup
t^{-1}-\chi_j(t)\psi_i(v_j) \overset{\eqref{eq:aj}}= \left(a_j(t)+\chi_j(t)c_j^{(i)} -c_j^ { (i) }
\right)g_j^ { -1 }.
\end{align*}
Thus $a_j(t)=(1-\chi_j(t))c_j^{(i)}$.
If
${\chi_j}_{|G(H)}\neq \varep$, then $c_j^{(1)}=c_j^{(2)}=\dfrac{a_j(t)}{1-\chi_j(t)}$
for $t\in G(H)$ with $\chi_j(t)\neq 1$. \eqref{lem:caracterizar f-a-item-b} $a_j=0$ by Proposition
\ref{prop:properties-Ak} \eqref{item-cc-prop:properties-Ak}.
\epf

Now, we study the set $\Alg^{\cH_{k}}(Y_k,\cA_k)$ for Step (1b) in the Strategy.

\begin{lema}\label{lem:caracterizar f-b}
Let $\varphi_i\in\Alg^{\cH_{k}}(Y_k,\cA_k)$ and such that
$\langle\varphi_i(Y_k^+)\rangle\neq
\cA_k$, $i=1,2$. Fix $j$, $1\leq j\leq n$, with
${\chi_j}_{|G(H)}\neq
\varep$.
\begin{enumerate}\renewcommand{\theenumi}{\alph{enumi}}
 \renewcommand{\labelenumi}{(\theenumi)}
  \item\label{lem:caracterizar f-b-item-a}
$\varphi_1(v_j)=\varphi_2(v_j)$.
 \item\label{lem:caracterizar f-b-item-b}
If $H$ is semisimple, then $\varphi_1(v_j)=\varphi_2(v_j)=\gamma(v_j)$.
 \end{enumerate}
\end{lema}
\pf
Notice that $t\varphi_i(v_j)t^{-1}-\chi_j(t)\varphi_i(v_j)\in \left\langle
\varphi_i(v_j)\right\rangle$. The computation in Lemma \ref{lem:caracterizar f-a} shows
that
$a_j(t)=(1-\chi_j(t))c_j^{(i)}$ since $\langle \varphi_i(v_j)\rangle\subseteq \langle
\varphi_i(Y_k)^+\rangle$.
\epf

\subsection{Some tools for diagonal braidings}\label{subsec:powers}

Let $V$ be a vector space with a basis $\xi_1,\dots,\xi_\theta$. Let $N_i \in \N$, $1\le i \le \theta$.
Let $\Omega=(\omega_{ij})_{1\leq i,j\leq \theta} \in \ku^{\theta\times \theta}$ such that $\omega_{ii}=1$ and $\omega_{ij}\omega_{ji} = 1$
for every $i, j$.
Consider the {\it quantum linear space} associated to $\Omega$, that is
\begin{align*}
\ku_{\Omega}[\xi_1,\dots,\xi_\theta]=T(V)/I_\Omega, \qquad \text{ for }\quad I_\Omega=\langle
\xi_i\xi_j-\omega_{ij}\xi_j\xi_i
\rangle_{1\leq i,j\leq \theta}.
\end{align*}

\medbreak

The following well-known lemma is useful to deal with the Strategy for
braidings of diagonal type.

\begin{lema}\label{lema:cociente power root vectors}
Fix $(\lambda_i)_{1\leq i\leq
\theta}\in\ku^\theta$ such that 
\begin{align}\label{eq:omega}
\lambda_i &= 0, & \text{when } \omega_{ij}^{N_i}&\neq 1 \text{ for some } j.
\end{align}
Let $S\subseteq I_\Omega$ and set $S' = S\cup \{\xi_i^{N_i}-\lambda_i\}_{1\leq i\leq
\theta}$. Then $T(V)/\langle S'\rangle
\neq 0$.
\end{lema}
\pf
It suffices to consider $S=I_\Omega$, as $T(V)/\langle S'\rangle \twoheadrightarrow T(V)/\langle I_\Omega'\rangle$. We can assume that $\lambda_i= 0$,
$i=1,\dots, k$ and $\lambda_i\neq 0$, if $i>k$, for some $0\leq k\leq \theta$. Set
$I=\langle I_\Omega\cup \{\xi_i^{N_i}\}_{1\leq i\leq
k}\rangle$; this is a proper $\N_0^\theta$-graded ideal and therefore the quotient
$\ku_\Omega[\xi_1,\dots,\xi_\theta]/\langle \xi_i^{N_i}\rangle_{1\leq i\leq
k}=T(V)/I\neq 0$. 

By \eqref{eq:omega} the elements $\xi_i^{N_i}$, $k+1\leq i\leq \theta$ are central in
$\ku_\Omega[\xi_1,\dots,\xi_\theta]$. Moreover, the subalgebra $P\subset T(V)/I$ generated by
their images is a polynomial algebra. We consider the 1-dimensional representation $M=\ku$ of
$P$ given by $\xi_i^{N_i}\cdot 1=\lambda_i$ and $M'=T(V)/I \ot_P M$ the induced
representation of $T(V)/I$. Notice that the algebra map $T(V)/I\to \End M'$ factors
through $T(V)/\langle I_\Omega'\rangle$ and hence this algebra is nonzero.
\epf

We will also make use of the following remark.

\begin{Rem}\label{rem:basis-quotient}
Let  $\{r_i\}_{i\in I}$ be a family of monomials in $ T(V)$ and set
$J=\langle r_i\rangle_{i\in I}$. Then (the image of) the set of
monomials in $T(V)$ that do not contain an $r_i$ as a subword is a linear basis of
$T(V)/J$. Indeed, a basis of $T(V)$ is given by the collection of all monomials. This set
can be splitted in two subsets: the monomials containing an $r_i$ as a subword and those
that do not. The first subset is a linear basis of $J$.
\end{Rem}

\subsection{An example of diagonal type}

Assume $\ku=\CC$. Let $\zeta\in\ku$  be a primitive $9^{th}$-root of unity. We apply our
strategy to
classify the liftings of the Nichols algebra
associated to the diagram
\begin{align}\label{eqn:diagram}
 &\xymatrix{ \,^{-\zeta} \circ\ar @<-0.01ex> @{-}[r]^{\zeta^7}  &
\circ\,^{\zeta^3}}
\end{align}
of \cite[Table 1, row 9]{He-adv}. Consider a matrix $(q_{ij})_{1\leq i,j\leq2}$
corresponding to \eqref{eqn:diagram}, that is $q_{11}=-\zeta$, $q_{22}=\zeta^3$ and
$q_{12}q_{21}=\zeta^7$. Let $\Gamma$ be a finite
group such that there is a
realization of this braiding, {\it i.e.} there are $g_1,g_2\in\Gamma$,
$\chi_1,\chi_2\in\widehat{\Gamma}$ with $\chi_j(g_i)=q_{ij}$, $1\leq i,j\leq2$. Set
$H=\ku \Gamma$ and $V\in\ydh$ the associated Yetter-Drinfeld module: $V$ has a
basis
$\{x_1,x_2\}$ with $x_i\in V^{\chi_i}_{g_i}$, $i=1,2$. Let
\begin{align}
\label{eqn:x12} x_{12}&=x_1x_2-q_{12}x_2x_1, & x_{112}&=x_1x_{12}-q_{11}q_{12}x_{12}x_1,
\\
\notag  x_{1112}&=x_1x_{112}-q_{11}^2q_{12}x_{112}x_1, &
x_{122}&=x_{12}x_2-q_{12}q_{22}x_2x_{12},
\\
\notag
x_{1,122}&=x_1x_{122}-q_{11}q_{12}^2x_{122}x_1.
&&
\end{align}
By \cite[Example 2.5]{Ang}, $\toba(V)$ is presented by generators $x_1,x_2$ and
relations $$x_1^{18}=x_2^3=x_{12}^{18}=x_{1112}=x_{1,122}-a\, x_{12}^2=0,$$
for $a=\zeta^7q_{12}(1+\zeta)^{-1}$. We fix the following stratification:
\begin{align}\label{eqn:stratification}
\Gc_0&=\{x_1^{18},x_2^3\}, & \Gc_1&=\{x_{1,122}-a\,
x_{12}^2\}, & \Gc_2&=\{x_{1112}\}, & \Gc_3&=\{x_{12}^{18}\}.
\end{align}
Set $\prov=\prov_4=\B(V)\# H$. Let $\lambda_1,\lambda_2 \in \ku$ be subject to:
\begin{align}\label{eqn:l1-l2}
&\lambda_1=0 \quad \text{if }\, \chi_1^{18}\neq\varep, && \lambda_2=0 \quad \text{if }\,
\chi_2^{3}\neq\varep.
\end{align}
Let $\prova(\lambda_1,\lambda_2)$ be the quotient of $\mT(V)$
by the ideal generated by
\begin{align}\label{eqn:no-moco}
& x_1^{18} - \lambda_1, && x_2^3 - \lambda_2,
&& x_{1,122}-a\, x_{12}^2, && x_{1112}.
\end{align}
\begin{Rem}\label{rem:A-no-cero}
$\prova(\lambda_1,\lambda_2)\neq 0$.
\end{Rem}

\pf
By Remark \ref{rem:comod-alg-boson} \eqref{item-a-rem:comod-alg-boson}, $\prova(\lambda_1,\lambda_2)\simeq T(V) / J \# H$
where $J$ is the ideal generated by the relations \eqref{eqn:no-moco}. Set $\xi_i = x_i$, $i = 1,2$, $\omega_{12} = q_{12} = \omega_{21}^{-1}$.
Then, in the notation of Lemma \ref{lema:cociente power root vectors}, $S := \{x_{1,122}-a\, x_{12}^2,  x_{1112}\} \subset I_{\Omega}$ by \eqref{eqn:x12}.
Condition \eqref{eq:omega} is tantamount to \eqref{eqn:l1-l2}, and $J = \langle S'\rangle$. Thus Lemma \ref{lema:cociente power root vectors} applies.
\epf

Let $\mL(\lambda_1,\lambda_2)$ be the quotient of $\mT(V)$
by the ideal generated by
\begin{align*}
& x_1^{18}-\lambda_1(1- g_1^{18}), && x_2^3-\lambda_2(1- g_2^3),
&& x_{1,122}-a\, x_{12}^2, && x_{1112}.
\end{align*}
Notice that this is a Hopf ideal, as $\Gc_0\cup \Gc_1\cup \Gc_2\subseteq \mP(T(V))$.

\smallbreak

We will now follow the strategy in Subsection \ref{subsec:strategy} in order to find all
the liftings of $\B(V)$ over $\Gamma$. We stick to the notation therein.

\begin{prop}\label{pro:galois-abeliana}
\begin{enumerate}\renewcommand{\theenumi}{\alph{enumi}}
\renewcommand{\labelenumi}{(\theenumi)}
\item $\prova(\lambda_1,\lambda_2)$ is a right $\prov_3$-Galois object with coaction
induced by the comultiplication in $\mT(V)$. Moreover,
\begin{align*}
 \Lambda_3=\big\{\prova(\lambda_1,\lambda_2):\, \lambda_1,\lambda_2 \text{ as in }
\eqref{eqn:l1-l2}\big\}.
\end{align*}
\item $L(\prova(\lambda_1,\lambda_2),\prov_3)\simeq\mL(\lambda_1,\lambda_2)$.
\item Let $\lambda_3\in\ku$ be subject to:
\begin{align}\label{eqn:l3}
 \lambda_3=0 \quad \text{if }\, \chi_1^{18}\chi_2^{18}\neq\varep.
\end{align}
Then the algebra $
\prova(\lambda_1,\lambda_2,\lambda_3):=\prova(\lambda_1,\lambda_2)/\langle
x_{12}^{18}-\lambda_3\rangle $
is a right $\prov$-Galois object.  Moreover,
\begin{align*}
 \qquad\Lambda_4=\big\{\prova(\lambda_1,\lambda_2,\lambda_3):\, \lambda_1,\lambda_2 \text{
as in } \eqref{eqn:l1-l2}\text{ and }\lambda_3 \text{ as in } \eqref{eqn:l3}\big\}.
\end{align*}
\end{enumerate}
\end{prop}
\pf
Following the Strategy in \ref{subsec:strategy}, we start by constructing the set
$\Lambda_1$. For this we use
(1b). Consider the subalgebra $Y_0$ of $\prov_0=\mT(V)$ generated by $x_1^{18}g_1^{-18}$
and $x_2^3g_2^{-3}$. This is a free associative algebra in two generators. Then we have
$\Alg^{\cH_0}(Y_0,\prova_0)\cong \ku^2$, since every map is determined by its value on
$x_1^{18}g_1^{-18}$ and $x_2^3g_2^{-3}$ and these values must be
$x_1^{18}g_1^{-18}-\lambda_1\,g_1^{-18}$
and $x_2^3g_2^{-3}-\lambda_2\,g_2^{-3}$, for some $\lambda_1,\lambda_2\in\ku$, by Remark
\ref{rem:caracterizar f}. Then $\Lambda_1$ is the set of all algebras
$\prova_1(\lambda_1,\lambda_2)$ obtained as $\mT(V)/\langle
x_1^{18} - \lambda_1,  x_2^3 - \lambda_2\rangle$, for $\lambda_1,\lambda_2$ subject to
\eqref{eqn:l1-l2} by Lemma \ref{lem:caracterizar f-b}. Indeed, these algebras are nonzero
since they project over
$\prova(\lambda_1,\lambda_2)$ which is nonzero by Remark
\ref{rem:A-no-cero}. We denote by $y_1, y_2$ the images of the generators $x_1,
x_2$ in each one of these quotients.

For $\Lambda_2$ we use again (1b). Set $Y_1=\ku\langle (x_{1,122}-a\,
x_{12}^2)g_1^{-2}g_2^{-2}\rangle\subset \prov_1$. Notice that
$\chi_1^2\chi_2^2(g_1^2g_2^2)=\zeta^8$
so $\chi_1^2\chi_2^2\neq\varep$. It follows that $Y_1$ is a polynomial algebra. Indeed, by
Lemma \ref{lem:X} we need to check that $z=\left(x_{1,122}-a\,
x_{12}^2\right)^9\neq 0$. Now, $z$ is a linear combination of monomials
containing $(x_1^2x_2^2)^9$ with coefficient 1. So it is nonzero by Remark
\ref{rem:basis-quotient}. Lemma \ref{lem:caracterizar
f-b}
implies that
\begin{align*}
\left(x_{1,122}-a\,x_{12}^2\right)g_1^{-2}g_2^{-2} & \mapsto
\gamma_1\left(\left(x_{1,122}-a\, x_{12}^2\right)g_1^{-2}g_2^{-2}\right),
\end{align*}
is the unique possible map in $\Alg^{\cH_1}(Y_1,\prova_1)$. Let $y_{1,122}$, $y_{12}\in
\prova_1$ be defined as in \eqref{eqn:x12}. It is easy to see that
\begin{align*}
\gamma_1\left(\left(x_{1,122}-a\,
x_{12}^2\right)g_1^{-2}g_2^{-2}\right)&=\left(y_{1,122}-a\,
y_{12}^2\right)g_1^{-2}g_2^{-2}.
\end{align*}
Indeed,
$\gamma_1\left(\left(x_{1,122}-a\,
x_{12}^2\right)g_1^{-2}g_2^{-2}\right)=\left(y_{1,122}-a\,
y_{12}^2\right)g_1^{-2}g_2^{-2}-c\, g_1^{-2}g_2^{-2}$, for some $c\in \ku$ by
$\cH_1$-colinearity but $c=0$ because $\gamma_1$ is  $H$-linear. Then $\Lambda_2$ is
composed of the algebras
$\prova_2(\lambda_1,\lambda_2)=\prova_1(\lambda_1,\lambda_2)/\langle y_{1,122}-a\,
y_{12}^2\rangle$ with $\lambda_1,\lambda_2$ subject to
\eqref{eqn:l1-l2}. These are nonzero as they project over
$\prova(\lambda_1,\lambda_2)$.

For $\Lambda_3$ we also use (1b). Set $Y_2=\ku\langle
x_{1112}g_1^{-3}g_2^{-1}\rangle\subset \prov_2$.
We have that $\chi_1^3\chi_2(g_1^3g_2)=-\zeta^6$,
so $\chi_1^3\chi_2\neq\varep$. As above, $Y_2$ is a polynomial algebra. Again, by
Lemma \ref{lem:X} it is enough to check that $x_{1112}^6\neq 0$. For this,
let $F$ be the set composed of the monomials $x_1^{18}$, $x_2^3$ and those appearing in
the
expression of $x_{1,122}-a\, x_{12}^2$. Set $J=\langle F\rangle\subset T(V)$, then there
exists a projection $\cH_2\twoheadrightarrow T(V)/J$. As $x_{1112}^6$ is a linear
combination of monomials containing $(x_1^3x_2)^6$ with coefficient 1, Remark
\ref{rem:basis-quotient} implies $x_{1112}^6\neq 0$ in  $T(V)/J$ and hence it is nonzero
in $\cH_2$. Thus Lemma \ref{lem:caracterizar f-b}
implies that $x_{1112}g_1^{-3}g_2^{-1}\mapsto\gamma_2(x_{1112}g_1^{-3}g_2^{-1})$ is the
unique possible map in
$\Alg^{\cH_2}(Y_2,\prova_2)$. Also, it is easy to see that
$\gamma_2(x_{1112}g_1^{-3}g_2^{-1})=y_{1112}g_1^{-3}g_2^{-1}$,
for $y_{1112}$ defined as in \eqref{eqn:x12}. We have already seen that the quotients
$\cA_2/\langle y_{1112}\rangle=\prova(\lambda_1,\lambda_2)$ are nonzero.
We obtain that $\Lambda_3$ is the set of all the algebras
$\prova(\lambda_1,\lambda_2)$ and (a) follows. Now (b) holds by 
Corollary \ref{cor:L-cociente}.

For $\Lambda_4$, we use (1a), since $\ku\langle x_{12}^{18}\rangle$ is a normal subalgebra
of $\cH_2$. This follows because $x_{12}^{18}$ is in the center of $\prov_3$, which can
be proved using \cite{GAP}, see also \cite{AAGI}. By Lemma \ref{lem:X},
$X_3$ is a
polynomial algebra. Using \cite{GAP} again\footnote{This coaction is computed with
\cite{GAP} using the method described in \cite[Appendix]{GIV}.},  we see that
\begin{align}\label{eqn:rho-y12}
 \rho_3(y_{12}^{18})=y_{12}^{18}\ot 1 + g_{1}^{18}g_2^{18}\ot x_{12}^{18}.
\end{align}
Hence, $\gamma_3(x_{12}^{18})=y_{12}^{18}+c$, for some $c\in\ku$. Now, if
$\chi_1^{18}\chi_2^{18}\neq \varep$, then $c=0$ and there is a unique
map in $\Alg_{\cH_3}^{\cH_3}(X_3,\cA_3)$, determined by
$x_{12}^{18}g_{1}^{-18}g_2^{-18}\mapsto y_{12}^{18}g_{1}^{-18}g_2^{-18}$ by Lemma
\ref{lem:caracterizar f-a}. On
the other hand, if $\chi_1^{18}\chi_2^{18}= \varep$, then it follows by Remark
\ref{rem:caracterizar f}  that $\Alg_{\cH_3}^{\cH_3}(X_3,\cA_3)\cong \ku$, since for each
$\lambda_3\in\ku$, $x_{12}^{18}g_{1}^{-18}g_2^{-18}\mapsto
y_{12}^{18}g_{1}^{-18}g_2^{-18}-\lambda_3\,g_{1}^{-18}g_2^{-18}$ induces an algebra
morphism $X_3\to \cA_3$ in $\mathcal{YD}_{\cH_3}^{\cH_3}$.
Hence  (c) follows.
\epf

Let $\mL(\lambda_1,\lambda_2,\lambda_3)$ be the quotient of $\mL(\lambda_1,\lambda_2)$ by
the ideal generated by
$$
x_{12}^{18}-\lambda_3(1-g_1^{18}g_2^{18}).
$$
In the next theorem we show that this is a Hopf ideal and that the family of
Hopf algebras $\mL(\lambda_1,\lambda_2,\lambda_3)$ exhausts the list of liftings of
$\B(V)$ over $\Gamma$. In particular, every lifting is a cocycle deformation of $\B(V)\#
\ku \Gamma$.
\begin{Thm}\label{teo:lifting-abeliano}
\begin{enumerate}\renewcommand{\theenumi}{\alph{enumi}}
\renewcommand{\labelenumi}{(\theenumi)}
\item $\mL(\lambda_1,\lambda_2,\lambda_3)$ is a cocycle deformation of $\prov$.
\item $\mL(\lambda_1,\lambda_2,\lambda_3)$ is a lifting of $\B(V)$ over $\Gamma$.
\item Reciprocally, if $L$ is a lifting of $\B(V)$ over $\Gamma$, then there are
$\lambda_1,\lambda_2,\lambda_3$ such that $L\simeq \mL(\lambda_1,\lambda_2,\lambda_3)$.
\end{enumerate}
\end{Thm}
\pf
(a) We use \cite{GAP} as in \eqref{eqn:rho-y12} to see that
\begin{align}\label{eq:x12-primitive}
\gamma_3(x_{12}^{18})_{(-1)}\ot\gamma_3(x_{12}^{18})_{(0)} - g_{1}^{18}g_2^{18}\ot
\gamma_3(x_{12}^{18})=x_{12}^{18}\ot 1.
\end{align}
Then $x_{12}^{18}$ satisfies \eqref{eqn:primitivo2}. Hence, 
$L(\prova(\lambda_1,\lambda_2,\lambda_3),\prov)\simeq
\mL(\lambda_1,\lambda_2,\lambda_3)$, by Corollary
\ref{cor:L-cociente}.

(b) follows by Proposition \ref{pro:deformations-liftings} (b) and (d).

(c) Let $\phi:\mT(V)\to L$ be a lifting map. If $r\in\Gc_0\cup\Gc_1\cup\Gc_2$, then $r$ is
$(g(r),1)$-primitive for $g(r)\in
\Gamma$, hence $\phi(r)\in L_1$. Let $\chi_r\in\widehat{\Gamma}$ be the character
from
the $\Gamma$-action on $r$. Now, the pair $(\chi_r,g(r))$ is different from
$(\chi_i,g_i)$, $i=1,2$ and thus $\phi(r)\in\ku\Gamma$ by Lemma
\ref{lema:liftings gral} (b), see also \cite[Lemma 6.1]{AS3}.
Indeed, $\chi_1^{18}(g_1^{18})=\chi_2^{3}(g_2^{3})=1$ and we have already seen
that $\chi_1^2\chi_2^2(g_1^2g_2^2)=\zeta^8$, $\chi_1^3\chi_2(g_1^3g_2)=-\zeta^6$. Then
there
exist $\lambda_1,\lambda_2\in\ku$
such that $\phi$ factorizes through $\mL(\lambda_1,\lambda_2)$.
By equation \eqref{eq:x12-primitive} and Corollary \ref{cor:L-cociente}, $x_{12}^{18}$ is
$(g_1^{18}g_2^{18},1)$-primitive in $\mL(\lambda_1,\lambda_2)$. Also,
$\phi(x_{12}^{18})\in \ku\Gamma$ again by Lemma \ref{lema:liftings gral} (b). Hence, there
exists $\lambda_3\in\ku$ such that
$\phi$ factorizes
through $\mL(\lambda_1,\lambda_2,\lambda_3)$ and induces an isomorphism since both
algebras have dimension $\dim\toba(V)|\Gamma|$.
\epf

\subsection{A question}\label{sec:question}
Set $\prova_k\in\Cleft(\prov_k)$. To find
$\prova_{k+1}\in\Cleft(\prov_{k+1})$ we can
either apply Theorem \ref{thm:gunther-teo4} or Theorem \ref{thm:gunther-extendido}.
As said in Subsection \ref{subsec:strategy} both alternatives present a hard computational
obstacle, namely the computation of $X_k$ or the checking of $\langle \varphi
(Y_{k}^+)\rangle \neq \cA_{k}$.
Hence we need an {\it intermediate G\"unther's Theorem} exploiting the benefits of both
alternatives. That  said, we collect from the examples enough evidence to change
alternative (1b) by
\begin{itemize}
\item[(1c)]
Compute $Y_k$ and then $\Alg^{\prov_k}_{H}(Y_k,\prova_k)$.
\end{itemize}
Actually, in many examples we see that  not only
$\varphi\in\Alg^{\prov_k}_{H}(Y_k,\prova_k)$ induces a nonzero algebra $\cA_{k+1}$ but
also that any
$\cA_{k+1}\in\Cleft(\cH_{k+1})$, and hence any $\psi\in \Alg_{\prov_k}^{\cH_k}(X_k,\cA_k)$
is determined by
$\varphi=\psi_{|Y_k}\in \Alg^{\cH_k}(Y_k,\cA_k)$. Furthermore, as $\cH_k$ is
$Y_k$-faithfully flat, see Corollary \ref{cor:flat},
\begin{enumerate}
\item $X_k$ is the subalgebra generated by $Y_k\cdot \cH_k$, see
Remark \ref{rem:xk=nyk}.
\end{enumerate}

\begin{question}\label{q:1}
Is there a general setting in which any $\varphi\in\Alg^{\prov_k}_{H}(Y_k,\prova_k)$
extends to $\psi\in \Alg_{\prov_k}^{\cH_k}(X_k,\cA_k)$ with $\psi_{|Y_k}=\varphi$?
\end{question}

Assume that $H$ is finite-dimensional and semisimple. Then evidence of a positive answer
is given by (1) above and the fact that
\begin{enumerate}
\item[(2)] $\cA_k$ is an injective object in
$\mathcal{YD}^{\cH_k}_{H}$, see Lemma \ref{lem:Ainjective}.
\end{enumerate}
So, any $\varphi\in\Hom^{\prov_k}_{H}(Y_k,\prova_k)$
extends to $\psi\in \Hom_{\prov_k}^{\cH_k}(X_k,\cA_k)$ with $\psi_{|Y_k}=\varphi$.

As a last word, we recall that:
\begin{enumerate}
\item[(3)] There is an $H$-linear section $\gamma_k:\cH_k\to
\cA_k$ with ${\gamma_k}_{|H}=\id_{H}$.
\end{enumerate}


\begin{thebibliography}{AEGN}

\bibitem[AAGI]{AAGI} {\sc Andruskiewitsch, N.}, {\sc Angiono I.}, {\sc Garc\'ia
Iglesias, A.}, in
preparation.

\bibitem[AFGV]{AFGV} {\sc Andruskiewitsch, N.}, {\sc Fantino, F.}, {\sc Garc\'ia, G. A.}, {\sc Vendramin, L.},
{\it On Nichols algebras associated to simple racks.} Groups, algebras and applications, 31--56, Contemp. Math., 537, Amer. Math. Soc., Providence, RI, 2011.

\bibitem[AC]{AC} {\sc Andruskiewitsch, N.}, {\sc Cuadra, J.}, {\it On the
structure of (co-Frobenius) Hopf algebras}. J. Noncommut. Geom. {\bf 7}, 83--104, (2013).

\bibitem[AG1]{AG1} {\sc Andruskiewitsch, N.}, {\sc Gra\~na, M.},
\emph{From racks to pointed Hopf algebras}, Adv. Math.
\textbf{178}, 177--243, (2003).

\bibitem[AG2]{AG2} \bysame, \emph{Examples of
liftings of Nichols algebras over racks}, AMA Algebra Montp.
Announc. (electronic), Paper {\bf 1},  (2003).


\bibitem[AN]{andrunatale} {\sc Andruskiewitsch N.}, {\sc Natale S.},
Counting arguments for Hopf algebras of low dimension, {\it
Tsukuba J. Math.} {\bf 25}, 187--201, (2001).

\bibitem[AS1]{AS1} {\sc Andruskiewitsch, N.}, {\sc Schneider, H.J.}, \emph{Lifting of
quantum linear spaces and pointed Hopf algebras of order $p^3$},
J. Algebra \textbf{209}, 658--691 (1998).

\bibitem[AS2]{AS2} \bysame,
\emph{Pointed Hopf algebras}, ``New directions in Hopf algebras'',
MSRI series Cambridge Univ. Press; 1--68 (2002).

\bibitem[AS3]{AS3} \bysame,
\emph{On the classification of finite-dimensional pointed Hopf
algebras}, Ann. Math. \textbf{171}, 375--417, (2010).

\bibitem[AV1]{AV} {\sc Andruskiewitsch, N.}, {\sc Vay, C.}, \emph{Finite dimensional
Hopf algebras over the dual group algebra of the symmetric group in three letters},
Comm. Alg. {\bf 39}, 4507--4517, (2011).

\bibitem[AV2]{AV2} \bysame, \emph{On a family of Hopf
algebras of dimension 72}, Bull. Belg. Math. Soc. Simon Stevin {\bf 19} 415--443, (2012).

\bibitem[An]{Ang} {\sc Angiono, I.}, \emph{Nichols algebras of unidentified diagonal
type}. Comm. Alg., to appear.



\bibitem[AM\c S]{ardizzonimstefan}{\sc Ardizzoni A., Menini C.}, {\sc Stefan D.}, A
Monoidal Approach to Splitting Morphisms of Bialgebras, {\it Trans. Amer. Math. Soc.} {\bf
359}, 991--1044, (2007).

\bibitem[BDR]{BDR} {\sc Beattie, M., D\u{a}sc\u{a}lescu, S.}, {\sc Raianu, S.},
{\em Lifting of Nichols algebras of type $B_2$}, Israel J. Math.
{\bf 132} (2002), 1--28.


\bibitem[D]{Di}
{\sc Didt, D.}, \emph{Pointed Hopf algebras and quasi-isomorphisms}, Algebr.
Represent. Theory \textbf{8}, 347--362, (2005).

\bibitem[DT1]{DT1} {\sc Doi, Y.}, {\sc Takeuchi, M.},
{\it Hopf-Galois extensions of algebras, the Miyashita-Ulbrich action and Azumaya
algebras}, J. Algebra {\bf 121}, 488--516, (1989).

\bibitem[DT2]{DT} \bysame, {\it Multiplication alteration by two-cocycles - the quantum
version},
Comm. Algebra {\bf 22}, 5715--5732, (1994).

\bibitem[EGNO]{EGNO}  {\sc Etingof, P.}, {\sc Gelaki, S.}, {\sc Nikshych, D.}, {\sc
Ostrik, V.}, \emph{Tensor categories}. Available at
\verb"http://www-math.mit.edu/~etingof/tenscat1.pdf".

\bibitem[FG]{FG} {\sc Fantino, F.}, {\sc Garc\'ia, G. A.}, {\it On pointed Hopf
algebras over dihedral groups}, Pacific J. Math. {\bf 252}, 69--91,  (2011).

\bibitem[GAP]{GAP} {\sc The GAP Group}, \emph{GAP --- Groups, Algorithms and
Programming}. Version \textbf{4.4.12}, (2008), \verb"http://www.gap-system.org".


\bibitem[GBNP]{GBNP}  {\sc Cohen, A. M.}, {\sc Gijsbers, D. A. H.}, \emph{GBNP
0.9.5 (Non-commutative Gr\"obner bases)}, \verb"http://www.win.tue.nl/~amc".

\bibitem[GGI]{GGI} {\sc Garc\'ia, G. A.}, {\sc Garc\'ia Iglesias, A.},
{\it Pointed Hopf algebras over $\s_4$}. Israel J. Math. {\bf 183},
417--444, (2011).

\bibitem[GM]{GaM} {\sc Garc\'ia, G. A.}, {\sc Mastnak, M.},
{\it Deformation by cocycles of pointed Hopf algebras over non-abelian groups},
\texttt{arXiv:1203.0957v1}.

\bibitem[GIM]{GIM} {\sc Garc\'ia Iglesias, A.}, {\sc Mombelli, M.}
{\it Representations of the category of modules over pointed Hopf algebras over $\s_3$
and $\s_4$}, Pacific J. Math. {\bf 252} 343--378, (2011).

\bibitem[GIV]{GIV} {\sc Garc\'ia Iglesias, A.}, {\sc Vay, C.}
{\it Finite-dimensional Pointed or Copointed Hopf
algebras over affine racks}, J. Algebra, to appear. \texttt{arXiv:1210.6396}.

\bibitem[GrM]{GM1} {\sc Grunenfelder, L.}, {\sc Mastnak, M.}, \emph{Pointed and
copointed Hopf algebras as cocycle deformations}, \texttt{arxiv:0709.0120v2}.

\bibitem[Gu]{G} {\sc G\"unther,R.}, {\it Crossed products for pointed Hopf algebras}.
Comm.  Algebra, \textbf{27}, 4389--4410,  (1999).

\bibitem[H]{He-adv} {\sc Heckenberger, I.} \emph{Classification of arithmetic root
systems}. Adv. Math. \textbf{220} (2009), 59--124.

\bibitem[He]{He} {\sc Helbig, M.} \emph{On the lifting of Nichols algebras},
Comm. Alg. {\bf 40} (2012), 3317-3351.

\bibitem[Kh]{Kh} \textsc{Kharchenko, V.}, \emph{A quantum analog of
the Poincare-Birkhoff-Witt theorem}, Algebra and Logic \textbf{38}, 259--276, (1999).

\bibitem[M]{M} {\sc Majid, S.}, {\it Foundations of Quantum Group Theory}, Cambr.
Univ. Press, (1995).

\bibitem[Ma1]{M1} {\sc Masuoka, A.}, \emph{On Hopf algebras with cocommutative
coradicals},
J. Algebra \textbf{22}, 451--466, (1991).

\bibitem[Ma2]{M0} \bysame, \emph{Defending the negated Kaplansky conjecture},
Proc. Amer. Math. Soc. \textbf{129} 3185--3192, (2001).

\bibitem[Ma3]{M3} \bysame,
\emph{The fundamental correspondences in super affine groups and super formal groups},
J. Pure Appl. Algebra {\bf 202}, 284--312, (2005).

\bibitem[Ma4]{M4} \bysame, {\it Abelian and non-abelian second cohomologies of quantized
enveloping algebras}, J.  Algebra, \textbf{320}, 1--47, (2008).

\bibitem[Mo]{Mo} {\sc Montgomery, S.}, {\it Hopf algebras and their actions on
rings}. CBMS Regional Conference Series in Mathematics {\bf 82}, Amer. Math. Soc., (1993).

\bibitem[MO]{MO} {\sc Majid, S., Oeckl, R.},{\it Twisting of Quantum Differentials and the Planck Scale Hopf
Algebra}, Commun. Math. Phys. {\bf 205} (1999), 617--655.


\bibitem[M\"u]{Mug-revuma} {\sc M\"uger, M.} \emph{Tensor categories: a selective
guided tour}, Rev. Un. Mat. Argentina  \textbf{51}, 95--163, (2010).

\bibitem[Ra]{R} {\sc Radford, D.},
\emph{Freeness (projectivity) criteria  for Hopf algebras over Hopf subalgebras},
J.\ Pure and Appl. Algebra {\bf 11}, 15--28, (1977).

\bibitem[S1]{S} {\sc Schauenburg, P.}, {\it Hopf bi-Galois extensions}, Comm. Algebra
\textbf{24}, 3797--3825, (1996).

\bibitem[S2]{S2} \bysame, {\it The structure of a Hopf algebra with a weak projection}, Algebr. Represent. Theory
\textbf{3}, 187--211, (2000)

\bibitem[Sc]{Sc} {\sc Schneider, H. J.}, {\it Normal basis and transitivity of
crossed products for Hopf algebras}, J. Algebra \textbf{152}, 289--312, (1992).


\bibitem[T1]{T1} {\sc Takeuchi, M.},
{\it Free Hopf algebras generated by coalgebras}, J. Math. Soc. Japan {\bf
22}, 561--582, (1971).

\bibitem[T2]{T2} \bysame, {\it Quotient spaces for Hopf algebras}, Comm. Algebra {\bf
 22}, 2503--2523, (1994).

\end{thebibliography}
\end{document}